\documentclass[11pt,a4paper,british]{article}
\usepackage{yfonts, amsmath, amscd, amsthm, latexsym, anysize, graphicx, makeidx, appendix, type1cm, eso-pic, color, wasysym, float, bbold, dsfont, xcolor, framed, mdframed, pzccal, anyfontsize, stmaryrd, euscript} 
\usepackage[colorlinks, pagebackref, pdfusetitle, urlcolor=black, citecolor=black, linkcolor=black, bookmarksnumbered, plainpages=false]{hyperref}
\usepackage{yfonts, amsmath, amscd, amsthm, amssymb, latexsym, anysize, graphicx, makeidx, appendix, type1cm, eso-pic, color, wasysym, float, xcolor, framed, mdframed, pzccal, anyfontsize, courier, babel, titlesec, dsfont}
\usepackage [all,pdf,cmtip]{xy}
\SetSymbolFont{largesymbols}{normal}{OMX}{iwona}{m}{n}
\usepackage[T1]{fontenc}

\usepackage{mathpazo} 
\makeindex
\marginsize{2.5cm}{2.5cm}{2cm}{2cm}
\graphicspath{{/Users/Ramses/Documents/Universitat/02_Thesis/Pictures/}} 

\makeatletter
\def\th@plain{\thm@notefont{} \itshape}
\def\th@definition{\thm@notefont{}\normalfont}
\makeatother

\DeclareFontFamily{U}{futm}{}
\DeclareFontShape{U}{futm}{m}{n}{
  <-> s * [1] fourier-bb
  }{}
\DeclareSymbolFont{Ufutm}{U}{futm}{m}{n}
\DeclareSymbolFontAlphabet{\mathbb}{Ufutm}
\mathchardef\mhyphen="2D 

\makeatletter 
\DeclareRobustCommand\widecheck[1]{{\mathpalette\@widecheck{#1}}}
\def\@widecheck#1#2{%
    \setbox\z@\hbox{\m@th$#1#2$}%
    \setbox\tw@\hbox{\m@th$#1%
       \widehat{%
          \vrule\@width\z@\@height\ht\z@
          \vrule\@height\z@\@width\wd\z@}$}%
    \dp\tw@-\ht\z@
    \@tempdima\ht\z@ \advance\@tempdima2\ht\tw@ \divide\@tempdima\thr@@
    \setbox\tw@\hbox{%
       \raise\@tempdima\hbox{\scalebox{1}[-1]{\lower\@tempdima\box
\tw@}}}%
    {\ooalign{\box\tw@ \cr \box\z@}}}
\makeatother

\renewcommand{\thefootnote}{\fnsymbol{footnote}}

\newtheorem{defn0}{Definition}[section]
\newtheorem{prop0}[defn0]{Proposition}
\newtheorem{thm0}[defn0]{Theorem}
\newtheorem{lemma0}[defn0]{Lemma}
\newtheorem{corollary0}[defn0]{Corollary}
\newtheorem{example0}[defn0]{Example}
\newtheorem{remark0}[defn0]{Remark}
\newtheorem{conjecture0}[defn0]{Conjecture}

\newenvironment{proposition}{\smallskip\begin{prop0}}{\end{prop0}}
\newenvironment{theorem}{\smallskip\begin{thm0}}{\end{thm0}}
\newenvironment{lemma}{\smallskip\begin{lemma0}}{\end{lemma0}}
\newenvironment{remark}{\smallskip\begin{remark0}}{\end{remark0}}

\newcommand{\CC}{\operatorname{\mathfrak{C}}}

\newcommand{\EE}{\operatorname{\mathfrak{E}}}
\newcommand{\FF}{\operatorname{\mathfrak{F}}}
\newcommand{\GG}{\operatorname{\mathfrak{G}}}

\newcommand{\KK}{\operatorname{\mathfrak{K}}}
\newcommand{\LL}{\operatorname{\mathfrak{L}}}
\newcommand{\MM}{\operatorname{\mathfrak{M}}}
\newcommand{\NN}{\operatorname{\mathfrak{N}}}

\newcommand{\RR}{\operatorname{\mathfrak{R}}}

\newcommand{\ob}{\operatorname{Obj}}
\newcommand{\comp}{\operatorname{\mathpzc{Comp}}}
\newcommand{\cob}{\operatorname{\mathpzc{Cob}}}

\newcommand{\module}{\operatorname{\mathpzc{Mod}}}

\newcommand{\symm}{\operatorname{\mathpzc{Sym}}}

\newcommand{\vect}{\operatorname{\mathpzc{Vect}}}

\newcommand{\topo}{\operatorname{\mathpzc{Top}}}

\newcommand{\aaa}{\operatorname{\mathpzc{A}}}
\newcommand{\bbb}{\operatorname{\mathpzc{B}}}
\newcommand{\ccc}{\operatorname{\widetilde{\mathpzc{C}}}}
\newcommand{\ddd}{\operatorname{\widetilde{\mathpzc{D}}}}
\newcommand{\eee}{\operatorname{\mathpzc{E}}}

\newcommand{\mmm}{\operatorname{\mathpzc{M}}}

\newcommand{\oo}{\operatorname{\widetilde{\mathpzc{O}}}}
\newcommand{\kkk}{\operatorname{\mathpzc{K}}}
\newcommand{\www}{\operatorname{\mathpzc{W}}}
\newcommand{\oc}{\operatorname{\widetilde{\mathpzc{OC}}}}

\newcommand{\dnklein}{\operatorname{\mathpzc{dnKlein}}}

\newcommand{\dnsymriem}{\operatorname{\mathpzc{dnSymRiemann}}}

\newcommand{\homo}{\operatorname{Hom}}

\newcommand{\id}{\operatorname{Id}}

\newcommand{\s}{\operatorname{\mathbb{S}}}
\newcommand{\K}{\operatorname{\mathbb{K}}}
\newcommand{\Z}{\operatorname{\mathbb{Z}}}

\newcommand{\C}{\operatorname{\mathbb{C}}}

\newcommand{\gcal}{\operatorname{\widetilde{\mathcal{G}}}}

\newcommand{\tr}{\operatorname{Tr}}

\newcommand{\h}{\operatorname{H}}

\newcommand{\op}{\operatorname{op}}

\newcommand{\ba}{\operatorname{Bar}}

\newcommand{\lan}{\operatorname{Lan}}

\renewcommand{\thefootnote}{\fnsymbol{footnote}}

\begin{document}

\title{On the Structure of Unoriented Topological Conformal Field Theories}
\author{Rams\`es Fern\`andez-Val\`encia}
\date{}

\maketitle


\begin{abstract}
\noindent
We give a classification of open Klein topological conformal field theories in terms of Calabi-Yau $A_\infty$-categories endowed with an involution. Given an open Klein topological conformal field theory, there is a universal open-closed extension whose closed part is the involutive variant of the Hochschild chains of the open part.
\end{abstract}
\vfill
\let\thefootnote\relax\footnotetext{\date{\today}}
\let\thefootnote\relax\footnotetext{Contact: ramses.fernandez.valencia@gmail.com}

\tableofcontents



\section{Introduction}

\subsection{Oriented and Klein TQFTs} 

The study of topological conformal field theories began with the works of Segal on conformal field theory \cite{Segal87}. Inspired by Segal's work, Atiyah gave a list of axioms for what he defines as a topological quantum field theory \cite{Atiyah88}, what represents a simpler version of a topological conformal field theory. It was Moore and Segal \cite{MoSe06} who first gave a precise definition of a topological conformal field theory and suggested the importance of its study.
\\\\
For a finite set, whose elements are called $D$-branes, $\Lambda$ let $\cob_\Lambda$ be the category whose class of objects are 1-manifolds (disjoint unions of circles and intervals) with boundary labelled by $D$-branes and with class of morphisms given by cobordism of these. Given a field $\K$ of characteristic zero, a 2-dimensional topological quantum field theory (henceforth a TQFT) is a symmetric monoidal functor $\FF:\cob_\Lambda\to\vect_{\K}$, where $\vect_{\K}$ is the category of $\K$-vector spaces. Let us consider a surface $\Sigma$ with boundary components labelled as open, closed or free. Open and free boundary components will correspond to intervals, whereas closed boundary components will correspond to circles. Depending on the boundary components, we can study open TQFTs if $\Sigma$ has only open and free boundary components; closed TQFTs if $\Sigma$ has only closed boundary components or open-closed TQFTs, where $\Sigma$ has open, closed and free boundary components.
\\\\
It is well known that:
\begin{enumerate}
\item The category of 2-dimensional open TQFTs is equivalent to the category of not-necessarily commutative Frobenius algebras (page 7 \cite{MoSe06}) and
\item the category of 2-dimensional closed TQFTs is equivalent to the category of commutative Frobenius algebras (Theorem 3.3.2 \cite{Kock03}).
\end{enumerate} 

If we change the morphisms in $\cob_\Lambda$ allowing not only oriented surfaces but unoriented surfaces, we obtain {\em Klein topological quantum field theories}. Closed Klein topological quantum field theories have been studied and classified in terms of Frobenius algebras endowed with extra structure coming from the extra generator one has to consider: the real projective plane $\mathbb{RP}^2$ with two holes \cite{TuTu06, AlNa06}. Open Klein topological quantum field theories are equivalent to non-commutative Frobenius algebras endowed with an involution \cite{thbraun12}. Open-closed Klein topological quantum field theories are completely described algebraically in terms of structure algebras \cite{AlNa06}.

\subsection{Oriented and Klein TCFTs} 

If we endow the morphisms of $\cob_{\Lambda}$ with a complex structure we can define a category $\mathpzc{OC}_{\Lambda}$ with the same class of objects of $\cob_{\Lambda}$ and where the arrows are given by singular chains on moduli spaces of Riemann surfaces. In this new setting it makes sense to work at a chain level so we can consider symmetric monoidal functors of the form $\FF:\mathpzc{OC}_{\Lambda}\to\comp_{\K}$, where $\comp_{\K}$ is the category of chain complexes over a field $\K$ of characteristic zero. Such a functor $\FF$, satisfying certain conditions, is called a \index{TCFT} 2-dimensional \textit{topological conformal field theory} (a TCFT henceforth). As in the TQFT setting, we talk about open, closed and open-closed TCFTs depending on the boundary components of the Riemann surfaces we work with.
Open TCFTs were classified by Costello \cite{Costello07} in terms of $A_\infty$-categories satisfying a Calabi-Yau condition. The work done by Costello relies on a ribbon graph decomposition of the moduli space of Riemann surfaces with marked points. Costello also gives a universal extension from open TCFTs to open-closed TCFTs and proves that the homology associated to the closed part of an open-closed TCFT is described in terms of the Hochschild homology of the Calabi-Yau $A_\infty$-category associated to its open part. 
\\\\
Costello's work was partially generalized to the unoriented setting, that is replacing Riemann surfaces with Klein surfaces, by Braun \cite{thbraun12}. In his work, Braun gives a decomposition of the moduli space of Klein surfaces in terms of M\"obius graphs, allowing him to state the classification of open Klein TCFTs in terms of involutive $A_\infty$-algebras  using techniques from operads theory.

\subsection{A closer look at topological conformal field theories} 

By endowing the morphisms in $\mathpzc{Cob}_{\Lambda}$ with a complex structure we can define a category $\mmm_{\Lambda}$ with the class of objects of $\mathpzc{Cob}_{\Lambda}$ and with class of arrows given by moduli spaces of Riemann surfaces.
\\\\
Let $\CC: \topo\to \comp_{\K}$ be the singular chains functor from topological spaces to chain complexes. As Riemann surfaces form moduli spaces, applying $\CC$ to the space of arrows of $\mmm_\Lambda$ yields a differential graded symmetric monoidal category $\mathpzc{OC}_\Lambda$ with $\ob(\mathpzc{OC}_\Lambda)=\ob(\mmm_\Lambda)$ and with class of arrows:
$$\homo_{\mathpzc{OC}_\Lambda}(a,b):=\CC\left(\homo_{\mmm_\Lambda}(a,b)\right).$$

Given a set of D-branes $\Lambda$, a 2-dimensional open-closed TCFT with set of D-branes $\Lambda$ is a pair $(\Lambda,\FF)$, where $\FF$ is a h-split symmetric monoidal functor $\FF:\mathpzc{OC}_\Lambda\to\comp_{\K}.$ As in the TQFT, we can consider just open and free boundary components in order to work not with $\mathpzc{OC}_\Lambda$ but with a subcategory $\mathpzc{O}_\Lambda$; or we can consider just closed boundary components in order to work with a subcategory $\mathpzc{C}_\Lambda$. Therefore, considering h-split symmetric monoidal functors $\FF:\mathpzc{O}_\Lambda\to\comp_{\K}$ will yield open TCFTs whilst considering h-split symmetric monoidal functors $\FF:\mathpzc{C}_\Lambda\to\comp_{\K}$ will return closed TCFTs. 
\\\\
Costello classified open TCFTs in terms of Calabi-Yau $A_\infty$-categories. An $A_\infty$-category $\mathpzc{C}$ consists of:
\begin{enumerate}
  \item A class of objects $\ob(\mathpzc{C})$;
  \item for each $c_1,c_2\in\ob(\mathpzc{C})$, a $\mathbb{Z}$-graded abelian group of homomorphisms $\homo_{\mathpzc{C}}(c_1,c_2)$;
  \item for all $n\geq 1$, composition maps 
  $$b_n:\homo_{\mathpzc{C}}(c_1,c_2)\otimes\cdots\otimes \homo_{\mathpzc{C}}(c_n,c_{n+1})\to \homo_{\mathpzc{C}}(c_1,c_{n+1})$$ of degree $n-2$ satisfying homotopy associativity conditions \cite{Costello07}.
\end{enumerate}

If for each $c\in\ob(\mathpzc{C})$ there exists an element $1_c\in\homo_{\mathpzc{C}}(c,c)$ of degree zero such that
\begin{enumerate}
\item $b_2(f\otimes 1_c)=f$ and $b_2(1_c\otimes g)=g$ for $f\in \homo_{\mathpzc{C}}(c',c)$ and $g\in \homo_{\mathpzc{C}}(c,c')$; 
\item for $0\leq i\leq n$, if $f_i\in\homo_{\mathpzc{C}}(c_i,c_{i+1})$ and $j=j+1$, then
$$b_n(f_0\otimes f_1\otimes\cdots\otimes 1_{c_j}\otimes\cdots\otimes f_{n-1})=0$$
\end{enumerate}
we say that the $A_\infty$-category $\mathpzc{C}$ is unital.
\\\\
A (unital) Calabi-Yau $A_\infty$-category is an $A_\infty$-category $\mathpzc{E}$ with a non-degenerate pairing of chain complexes 
$\langle - , - \rangle_{e_1,e_2}: \homo_{\mathpzc{E}}(e_1,e_2)\otimes \homo_{\mathpzc{E}}(e_2,e_1)\to \K,$
satisfying certain conditions \cite{Costello07}.
\\\\
Costello proves in Lemma 7.3.4 \cite{Costello07} that the category of open TCFTs is quasi-equivalent to the category of unital Calabi-Yau $A_\infty$-categories. The way he proves this result is heavily based on a ribbon graph decomposition for the moduli space of Riemann surfaces \cite{Costello06}, what allows one to replace $\mathpzc{O}_\Lambda$ with another category which we can describe by a set of generators and relations.
\\\\
The results obtained by Costello are all twisted by a local system of coefficients on the moduli spaces which has been ignored here. This twisting is useful and necessary to Costello due to his motivations related to Gromov-Witten theory; we ignore it for the sake of simplicity in the notations: all the results contained in this manuscript hold if we keep track of this local system. 

\subsection{The results of this research} 

By extending Costello's techniques to the unoriented setting, the research developed here represents a completion of the picture started by Braun. The main result is:

\begin{theorem}\label{thetheorem}
\begin{enumerate}
\item There is a homotopy equivalence between open Klein TCFTs and Calabi-Yau $A_\infty$-categories endowed with an involution.
\item Given an open Klein TCFT, a universal open-closed extension to open-closed Klein TCFTs exists.
\item The homology of the closed part of the above open-closed TCFT is described in terms of the involutive Hochschild homology of its open part.
\end{enumerate}
\end{theorem}

The description of involutive Hochschild homology has been studied in detail in \cite{FeGi15}. Involutive Hochschild homology and ``usual'' Hochschild homology do not coincide unless the algebras involved are commutative and endowed with the trivial involution. 


\section{Homological algebra and category theory}

\textit{Braun \cite{thbraun12} gives a classification of open Klein topological conformal field theories in terms of Calabi-Yau $A_\infty$-categories endowed with involution using algebras over modular operads. It will be necessary to begin with an introduction of the concepts and notations which will be used henceforth and that will be central in these notes.}

\subsection{Modules over categories} 


Henceforth, all the categories will be differential graded symmetric monoidal categories (DGSM for short), and all the functors will be assumed to be differential graded functors. For DGSM categories $(\aaa,\otimes,1_{\aaa})$ and $(\bbb,\otimes, 1_{\bbb})$ a \textit{monoidal functor} is given by a triple $(\FF,F_0,F_1)$ where:
\begin{enumerate}
\item $\FF$ is an ordinary functor $\FF:\aaa\to\bbb$;
\item for objects $a_1,a_2\in\ob(\aaa)$ we have morphisms $F_1(a_1,a_2):\FF(a_1)\otimes \FF(a_2)\to \FF(a_1\otimes a_2)$ in $\bbb$ which are naturalt in $a_1$ and $a_2$;
\item for the units $1_{\aaa}$ and $1_{\bbb}$, we have a morphism in $\bbb$ of the form $F_0:1_{\bbb}\to \FF(1_{\aaa})$.
\end{enumerate}
Furthermore, the following diagrams must be commutative, for objects $a_1, a_2, a_3\in\ob(\aaa)$:
\[
\xymatrix{
\FF(a_1)\otimes (\FF(a_2)\otimes \FF(a_3))\ar[r]^-\cong\ar[d]_{\id\otimes F_1} & (\FF(a_1)\otimes \FF(a_2))\otimes\FF(a_3)\ar[d]^-{F_1\otimes\id} \\
\FF(a_1)\otimes\FF(a_2\otimes a_3)\ar[d]_-{F_1} & \FF(a_1\otimes a_2)\otimes\FF(a_3)\ar[d]^-{F_1} \\
\FF(a_1\otimes(a_2\otimes a_3))\ar[r]_-{\cong} & \FF((a_1\otimes a_2)\otimes a_3)
}
\]

\[
  \begin{minipage}[b]{0.45\linewidth}
  \centering
  \xymatrix
  { 
   \FF(a_2)\otimes 1_{\bbb} \ar[r]^-{\cong} \ar[d]_-{\id\otimes F_0} & \FF(a_2) \\
   \FF(a_2)\otimes \FF(1_{\aaa}) \ar[r]_-{F_1} & \FF(a_2\otimes 1_{\aaa})\ar[u]_-{\cong}
  }
  \end{minipage}
  \hspace{0.5cm}
  \begin{minipage}[b]{0.45\linewidth}
  \centering
  \xymatrixcolsep{3pc}\xymatrix
  {
  1_{\bbb}\otimes \FF(a_2) \ar[r]^-{\cong} & \FF(a_2) \\
  \FF(1_{\aaa})\otimes \FF(a_2)\ar[u]^-{F_0\otimes \id} \ar[r]_-{F_1} & \FF(1_{\aaa}\otimes a_2) \ar[u]_-{\cong}
  }
  \end{minipage}
\]

A monoidal functor $(\FF,F_0,F_1):\aaa\to\bbb$ between DGSM categories $\aaa$ and $\bbb$ will be called \textit{symmetric} if the following diagram commutes:
\[
\xymatrix{
\FF(a_1)\otimes \FF(a_2) \ar[r]^-{\sigma_{\bbb}}\ar[d]_-{F_1(a_1,a_2)} & \FF(a_2)\otimes \FF(a_1) \ar[d]^--{F_1(a_2,a_1)} 
\\
\FF(a_1\otimes a_2) \ar[r]_-{\FF(\sigma_{\aaa})} & \FF(a_2\otimes a_1)
}
\]
where $\sigma_{\aaa}$ and $\sigma_{\bbb}$ are the symmetry isomorphisms.
\\\\
A monoidal functor $(\FF,F_0,F_1)$ will be called \textit{split} if $F_1(a_1,a_2)$ and $F_0$ are isomorphims. We call $(\FF,F_0,F_1)$ \textit{h-split} if the corresponding homology morphisms $\h(F_1(a_1,a_2))$ and $\h(F_0)$ are isomorphisms.
\\\\
For two DGSM categories $\aaa$ and $\bbb$ and split monoidal functors $\MM,\NN: \aaa\to \bbb,$ 
a \index{Monoidal nat. transformation} \textit{monoidal natural transformation} $\phi:\MM\to \NN$ consists of a collection of maps $\phi_a$, for objects $a\in\ob(\aaa)$, in $\homo_{\bbb}(\MM(a), \NN(a))$ making the following diagrams commute:
\[
  \begin{minipage}[b]{0.45\linewidth}
  \centering
  \xymatrix
  { 
   \MM(a_1) \ar[r]^-{\phi_{a_1}} \ar[d]_-{\MM(f)} & \NN(a_1) \ar[d]^-{\NN(f)} \\
   \MM(a_2) \ar[r]_{\phi_{a_2}} & \NN(a_2)
  }
  \end{minipage}
  \hspace{0.5cm}
  \begin{minipage}[b]{0.45\linewidth}
  \centering
  \xymatrixcolsep{3pc}\xymatrix
  {
  \MM(a_1)\otimes \MM(a_2) \ar[r]^-{\phi_{a_1}\otimes \phi_{a_2}}\ar[d]_-\cong & \NN(a_1)\otimes \NN(a_2) \ar[d]^-\cong \\
  \MM(a_1\otimes a_2) \ar[r]_-{\phi_{a_1\otimes a_2}} & \NN(a_1\otimes a_2)
  }
  \end{minipage}
\]
  for morphisms $f:a_1 \to a_2$ and objects $a_1,a_2\in \ob(\aaa)$.
\\\\
For a DGSM category $\aaa$ and a field $\K$, a \index{Left $\aaa$-module} \textit{left $\aaa$-module} is a split symmetric monoidal functor $\LL:\aaa\to \comp_{\K}$.
A \index{Right $\aaa$-module} \textit{right $\aaa$-module} is a split symmetric monoidal functor $\RR:\aaa^{\op}\to\comp_{\K}.$ We have two categories, one of left $\aaa$-modules, denoted by $\aaa\hspace{-1pt}\mhyphen\hspace{-1pt}\module$, and another one of right $\aaa$-modules, which will be denoted by $\module\hspace{-1pt}\mhyphen\hspace{-1pt}\aaa$. An $\aaa-\bbb$-bimodule split symmetric monoidal functor $\FF:\aaa\otimes\bbb^{\op}\to\comp_{\K}$.

\subsection{Derived tensor products and push-forwards} 

Let $\MM$ be a $\bbb\mhyphen\aaa$-bimodule and $\NN$ a left $\aaa$-module. We define the left $\bbb$-module $\MM\otimes_{\aaa} \NN$ by saying that $(\MM\otimes_{\aaa} \NN)(b)$ is the complex with maps 
$\MM(b,a)\otimes_{\K}\NN(a)\to (\MM\otimes_{\aaa}\NN)(b)$
such that make the diagram commute for each pair $a,a'\in\ob(\aaa)$:
\begin{displaymath}
\xymatrix{
\MM(b,a)\otimes_{\K}\homo_{\aaa}(a',a)\otimes_{\K}\NN(a') \ar[r]^-{(1)} \ar[d]_-{(2)} & \MM(b,a)\otimes_{\K}\NN(a) \ar[d]\\
\MM(b,a')\otimes_{\K}\NN(a') \ar[r] & (\MM\otimes_{\aaa} \NN)(b)
}
\end{displaymath}
Where maps (1) and (2) denote left and right compositions.
\\\\
If $\mathfrak{f}$ is a functor between involutive DGSM categories $\mathfrak{f}:\aaa\to \bbb$, then $\homo_{\bbb}$ is a $\bbb$-bimodule and becomes an $\aaa\mhyphen\bbb$-bimodule and a $\bbb\mhyphen\aaa$-bimodule via the functors
\begin{displaymath}
\left.\begin{array}{cccc}\homo_{\bbb}: & \bbb\otimes\bbb^{\op} & \to & \comp_{\K} \\ & b_1\otimes b_2 & \rightsquigarrow & \homo_{\bbb}(b_1,b_2)\end{array}\right.,
\end{displaymath}
\begin{displaymath}
\left.\begin{array}{cccc}\FF: & \aaa\otimes\bbb^{\op} & \to & \bbb\otimes \bbb^{\op} \\ & a\otimes b & \rightsquigarrow & \mathfrak{f}(a)\otimes b\end{array}\right.,
\end{displaymath}
\begin{displaymath}
\left.\begin{array}{cccc}\GG: & \bbb\otimes\aaa^{\op} & \to & \bbb\otimes \bbb^{\op} \\ & b\otimes a & \rightsquigarrow & b\otimes\mathfrak{f}(a)\end{array}\right. .
\end{displaymath}
\noindent
We define a functor $\mathfrak{f}_{\star}:\aaa\mhyphen\module\to \bbb\mhyphen\module$ by setting
\begin{displaymath}
\mathfrak{f}_{\star}(\MM):=\homo_{\bbb}\otimes_{\aaa}\MM=:B\otimes_{\aaa}\MM.
\end{displaymath}
We define a functor $\mathfrak{f}^{\star}:\bbb\mhyphen\module\to \aaa\mhyphen\module$ as the composition of $\NN:\bbb\to\comp_{\K}$ with $\mathfrak{f}:\aaa\to\bbb$:
\begin{displaymath}
\left.\begin{array}{cccccc}\mathfrak{f}^{\star}: & \aaa & \to & \bbb & \to & \comp_{\K} \\ & a & \rightsquigarrow & \mathfrak{f}(a) & \rightsquigarrow & \NN(\mathfrak{f}(a))\end{array}\right. .
\end{displaymath}

Let us denote by $S_n$ the symmetric group on $n$ letters. For $\aaa$ an involutive DGSM category let $\symm{\aaa}$ be the subcategory whose objects are those of $\aaa$ and whose morphisms are the identity maps and the symmetry isomorphisms:
\begin{displaymath}
a_1\otimes a_2\otimes \cdots \otimes a_n\cong a_{\sigma(1)}\otimes \cdots \otimes a_{\sigma(n)}, \mbox{ for } \sigma\in S_n. 
\end{displaymath}

We define the category $\symm_{\K}\aaa$ as the sub-linear category of $\aaa$ whose morphisms are spanned by the morphisms in $\symm\aaa$.
\\\\
Following \cite{Costello07}, we denote by $\comp_{\K}^\Delta$ the symmetric monoidal category of simplicial chain complexes. The realization of  $C\in\ob\left(\comp_{\K}^\Delta\right)$ is $|C|:=\bigoplus_{n\geq 0}\frac{C\{n\}}{C^{\text{deg}}\{n\}}[-n],$ where $C^{\text{deg}}\{n\}$ is the image of the degeneracy maps and $[-n]$ denotes a degree shifting.
\\\\
Given an $\aaa\mhyphen\bbb$-bimodule $\MM$ and a left $\bbb$-module $\NN$, we define the left $\aaa$-module:
\begin{displaymath}
\MM\otimes_{\bbb}^{\mathbb{L}}\NN:=\MM\otimes_{\bbb}\ba_{\bbb}(\NN), 
\end{displaymath}
where $\ba_{\bbb}(\NN):=\left|\ba_{\bbb}^{\Delta}(\NN)\right|$ and $\ba_{\bbb}^{\Delta}(\NN)$ is the following simplicial $\bbb$-module:
\begin{displaymath}
\left(\ba_{\bbb}^{\Delta}(\NN)\right)[n]:=\underbrace{\bbb\otimes_{\symm_{\K}\bbb}\bbb\otimes_{\symm_{\K}\bbb} \cdots \otimes_{\symm_{\K}\bbb}\bbb}_{n \text{ times}} \otimes_{\symm_{\K}\bbb} \NN.
\end{displaymath}

The face maps come from the product maps $\bbb\otimes_{\symm_{\K}\bbb} \bbb\to \bbb$ whilst the degeneracy maps come from the maps $\symm_{\K}\bbb \to \bbb$.
\\\\
An $\aaa$-module $\MM$ is \index{Flat module} \textit{flat} if the functor $(-)\otimes_{\aaa}\MM:\module\mhyphen\aaa\to \comp_{\K}$ is exact, that is: if it sends quasi-isomorphisms to quasi-isomorphisms. We denote by $\aaa\mhyphen\mathpzc{flat}$ the full subcategory of flat $\aaa$-modules and the inclusion by $\mathfrak{i}:\aaa\mhyphen\mathpzc{flat} \hookrightarrow \aaa\mhyphen\module$. 
\\\\
The definition for the derived tensor product makes sense due to the following Lemmata:

\begin{lemma}[Lemma 4.3.3 \cite{Costello07}]
The projection $\pi:\ba_{\bbb}(\NN)\to\NN$ is a quasi-isomorphism.
\end{lemma}


\begin{lemma}[Lemma 4.3.4 \cite{Costello07}]
For any $\bbb$-module $\NN$, $\ba_{\bbb}(\NN)$ is a flat $\bbb$-module.
\end{lemma}


For $\mathfrak{f}:\aaa\to\bbb$ a functor between involutive DGSM categories and $\NN$ a left $\aaa$-module we define
\begin{equation}\label{derivedtensor}
\mathbb{L}\mathfrak{f}_\star \NN:=\bbb\otimes_{\aaa}^{\mathbb{L}}\NN.
\end{equation}

\begin{remark}
Let us recall that 
$$\mathbb{L}\mathfrak{f}_\star \NN:=\bbb\otimes_{\aaa}^{\mathbb{L}}\NN\simeq\bbb\otimes_{\aaa}\ba_{\aaa}\NN=\bbb\otimes_{\aaa}\underbrace{\left|\ba^\Delta_{\aaa}\NN\right|}_{\EE}$$ and it is well known that we can write the last tensor product as the coend $\int^{\aaa}\bbb\odot\EE$. On the other hand, for $\NN\in\ob(\aaa\mhyphen\module)$ and $\FF:\aaa\to\bbb$ a functor between involutive DGSM categories, we can write (Theorem 1, chapter X, section 4 \cite{MacLane98}), for each $c\in\ob(\bbb)$:
$$(\lan_{\FF}\NN)(c)=\int^{a\in\ob(\aaa)}\homo_{\bbb}(\FF(a),c)\odot\NN(a).$$
Then we can think of (\ref{derivedtensor}) as an example of a derived left Kan extension.
\end{remark}

\subsection{Quasi-isomorphisms in a category} 

A morphism $f:C_\bullet\to D_\bullet$ of complexes in an abelian category $\aaa$ is a \index{Quasi-isomorphism} \textit{quasi-isomorphism} if the corresponding homology morphism 
$$\h_n(f):\h_n(C_\bullet)\to\h_n(D_\bullet)$$ 
is an isomorphism for each $n\in \Z$.
\\\\
A category $\mathpzc{C}$, not necessarily abelian, has a notion of quasi-isomorphism when we are given a subset of $\homo_{\mathpzc{C}}$ which is closed under composition and contains all isomorphisms. Objects in $\mathpzc{C}$ are said to be \index{Quasi-isomorphic objects} \textit{quasi-isomorphic} if they can be connected by a chain of quasi-isomorphisms. We write $c_1\simeq c_2$ when two objects $c_1,\, c_2$ are quasi-isomorphic.
\\\\
We define a natural transformation $\phi$ between exact functors $\FF$ and $\GG$ as a quasi-isomorphism $\phi_c: \FF(c)\to \GG(c)$ is a quasi-isomorphism for every object $c\in \ob(\mathpzc{C})$.
\\\\
Given categories $\mathpzc{C}$ and $\mathpzc{D}$ with the notion of quasi-isomorphism, we define a \index{Quasi-equivalence} \textit{quasi-equivalence} as a pair of functors $\FF:\mathpzc{C}\to \mathpzc{D}$ and $\GG:\mathpzc{D} \to \mathpzc{C}$ such that the following quasi-isomorphisms of functors hold: $\FF\circ \GG\simeq \mathbb{1}_{\mathpzc{D}}$ and $\GG\circ \FF\simeq \mathbb{1}_{\mathpzc{C}}$. 

\begin{lemma}[cf. Lemma 4.4.1 \cite{Costello07}]
Given two involutive DGSM categories $\aaa$ and $\bbb$, let us assume that the homology functor $\h_\bullet(\FF):\h_{\bullet}(\aaa)\to \h_{\bullet}(\bbb)$ is fully faithful. Then the functor $\FF^{\star}\mathbb{L}\FF_{\star}$ is quasi-isomorphic to $\mathbb{1}_{\aaa\mhyphen\module}$.
\end{lemma}


\begin{theorem}[cf. Lemma 4.4.3 \cite{Costello07}]
For involutive DGSM categories $\aaa$ and $\bbb$, if $\FF:\aaa\to\bbb$ is a quasi-isomorphism, then the functors $\mathbb{L}\FF^{\star}$ and $\FF^{\star}$ are inverse quasi-equivalences between $\aaa\mhyphen\module$ and $\module\mhyphen\bbb$.
\end{theorem}


\begin{proposition}[Lemma 4.4.4 \cite{Costello07}]\label{costello444}
Let us consider $\FF^\star$ and $\mathbb{L}\FF_\star$ the induced quasi-equivalences between $\module\hspace{-2pt}\mhyphen \aaa\times \aaa\mhyphen\module\leftrightarrow \module\hspace{-2pt}\mhyphen \bbb\times \bbb\mhyphen\module$. Then the diagram below commutes up to quasi-isomorphim:
\[
\xymatrix{
\module\hspace{-2pt}\mhyphen \aaa\times \aaa\mhyphen\module \ar[r]^-{\otimes_{\aaa}^{\mathbb{L}}} \ar@<-3pt>[d]_-{\FF^\star} & \comp_{\K} \\
\module\hspace{-2pt}\mhyphen \bbb\times \bbb\mhyphen\module \ar[ur]_-{\otimes_{\bbb}^{\mathbb{L}}} \ar@<-3pt>[u]_-{\mathbb{L}\FF_\star} & 
}
\]
\end{proposition}


\section{Fundamentals from graph theory}

\textit{The role played by graphs is central in the theory of moduli spaces of Riemann or Klein surfaces as ribbon graphs provide orbi-cell decompositions of moduli spaces of Riemann surfaces \cite{Costello04, Costello06}. 
In order to deal with Klein surfaces, ribbon graphs are not enough and one has to introduce the concept of M\"obius graph. M\"obius graphs provide an orbi-cell decomposition of moduli spaces of Klein surfaces. For further details we refer to \cite{thbraun12}.}

\subsection{Ribbon graphs} 

\nocite{Giansiracusa11} 
A \index{Finite graph} \textit{finite graph} $\gamma$ consists of:
\begin{enumerate}
\item Finite sets of vertices $V(\gamma)$ and half-edges $H(\gamma)$;
\item an involution $\iota: H(\gamma)\to H(\gamma)$ and a map $\lambda: H(\gamma)\to V(\gamma)$.
\end{enumerate}

Given a finite graph $\gamma$, we say that two half-edges $a,b$ form an edge if $\iota(a)=b$; a half-edge $a$ is connected to a vertex $v$ if $\lambda(a)=v$. A leg $l$ in $\gamma$ is a univalent vertex; an external edge $e=(e_1,e_2)$ is an edge that meets a leg.
An internal edge is an edge for which neither end is univalent. A corolla is a graph consisting of a single vertex with several legs connected to it.
\\
\begin{figure}[h]
\centering
\scalebox{.8}{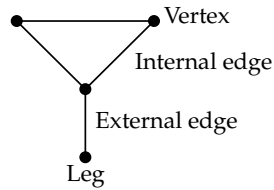}
\caption{A graph with four vertices, one leg, one external edge and three internal edges.}
\end{figure}

\begin{remark}
We can imagine and edge $e$ as a pair of half-edges $e_1,e_2$ by cutting $e$ in half. Observe that the involution $\iota$ swaps the half-edges. On the other hand $\lambda$, by sending a half-edge $e_i$ to a vertex $v_i$, is gluing $e_i$ to $v_i$. 
\end{remark}



Given two finite graphs $\gamma_1,\, \gamma_2$, a graph isomorphism $g:\gamma_1\to \gamma_2$ is given by a pair $(g_1,g_2)$ of bijections  $g_1: V(\gamma_1)\to V(\gamma_2)$ and $g_2: H(\gamma_1)\to H(\gamma_2)$ satisfying $\lambda\circ g_2= g_1\circ \lambda$ and $\iota\circ g_2=g_1\circ \iota$.
\\\\
A \index{Ribbon graph} \textit{ribbon graph} is a finite graph equipped with a cyclic ordering of the half-edges at each vertex and a labelling of the legs, that is:  the $n$ legs of the ribbon graph $\gamma$ are labelled by the elements of $\{1,\dots,n\}$. An isomorphism of ribbon graphs is an isomorphism of graphs that preserves the cyclic ordering at each vertex and the labelling of the legs.
\\\\
Given a ribbon graph $\gamma$ and an internal edge $e$ which is not a loop, we define the edge contraction $\gamma/e$ by endowing the graph $\gamma/e$, obtained after contracting the edge $e$, with the obvious cyclic ordering coming from the cyclic orderings at the vertices defining $e$.
\\\\
For a ribbon graph $\gamma$ and two internal edges $e_1,\,e_2$ that are not loops we have the following isomorphism: 
$(\gamma/e_1)/e_2\cong (\gamma/e_2)/e_1,$ assuming both sides are defined.
\\\\
A \index{Reduced ribbon graph} \textit{reduced ribbon graph} is a ribbon graph where each vertex is either univalent or has valence at least 3. Given a graph with at least one vertex having valence at least 3, we can associate to it reduced graphs by repeatedly contracting an edge attached to a vertex of valence 2 until the graph is reduced. 

\subsection{M\"obius graphs} 

A \index{M\"obius graph} \textit{M\"obius graph} is a ribbon graph $\gamma$ with a colouring of the half-edges by two colours, which means that we have a map $c:H(\gamma)\to \mathbb{Z}_2$. An isomorphism of M\"obius graphs is an isomorphism of graphs preserving the sum (modulo 2) of the labellings on each edge such that, at each vertex $v$, it can happen that either:
\begin{enumerate}
  \item The map preserves the cyclic ordering at $v$ and the colouring of the half-edges at $v$; or
  \item the map reverses the cyclic ordering at $v$ and reverses the colouring at the half-edges connected to $v$.
\end{enumerate} 

There is a very convenient way to visualize M\"obius graphs, and therefore ribbon graphs, which is based on thickening. If we thicken a graph in a way that the vertices become intervals and the edges become strips, it is not hard to see that we can get a ``surface'' from our graph. Now, the colouring in a M\"obius graph works as follow: let us consider a graph being a single edge where the two ends meet at a bivalent vertex; we can colour the half-edges compatibly or incompatibly; if do colour them in a compatible way, we get an annulus; if we colour the edges incompatibly, we get a M\"obius band.
\\\\
Given a M\"obius graph $\gamma$ and an internal edge $e$ (not being a loop) where both the half-edges of $e$ have the same colour, we define the graph contraction $\gamma/e$ as we did for ribbon graphs. This is well defined on isomorphism classes and can be extended to all internal edges but loops, regardless the colouring. For a M\"obius graph $\gamma$ and two internal edges $e_1,\, e_2$ (which are not loops) whose half-edges have the same colouring, the following isomorphism holds: $(\gamma/e_1)/e_2\cong (\gamma/e_2)/e_1,$ assuming both sides are defined.
\\\\
A \index{Reduced M\"obius graph} \textit{reduced M\"obius graph} is a M\"obius graph where each vertex is either univalent or has valence at least 3. Given a M\"obius graph with at least one vertex having valence at least 3, we can associate to it reduced M\"obius graphs by repeatedly contracting an edge attached to a vertex of valence 2 until the graph is reduced. 


\section{Fundamentals on Klein surfaces}

\textit{We revisit the concepts of Klein and nodal Klein surfaces and state equivalences of categories between Klein surfaces and Riemann surfaces with an involution following the results and techniques developed in \cite{thbraun12}. These equivalences will establish a duality that will make the forthcoming results almost a direct consequence of the results in \cite{Costello04, Costello06, Costello07}.} 

\subsection{Klein surfaces and symmetric Riemann surfaces} 

Let $D\subset \C$ be a non-empty open subset and $f:D\to\C$ a smooth map. We say that $f$ is \index{Dianalytic map} \textit{dianalytic} if its restriction to each component of $D$ is either analytic or anti-analytic. If $A$ and $B$ are non-empty subsets of the complex upper half-plane $\C^+$, a map $g:A\to B$ is called analytic (resp. dianalytic) on $A$ if it extends to an analytic (resp. dianalytic) map $g':U\to\C$ where $U$ is an open neighbourhood of $A$ in $\C$.

\begin{remark}
A surface, unless otherwise stated, is a connected and compact 2-dimensional manifold, possible with boundary. An atlas $\Xi$ on a surface $K$ is dianalytic if all the transition maps of $\Xi$ are dianalytic. A dianalytic structure on $K$ is a maximal dianalytic atlas. A \index{Klein surface} \textit{Klein surface} is a surface equipped with a dianalytic structure. A \index{Singular surface} \textit{singular} topological surface $(X,N)$ is a Hausdorff space $X$ with a discrete set $N\subset X$ of general singularities such that $X-N$ is a topological surface. Henceforth, we will consider these surfaces compact and possibly with boundary, where the boundary is defined to be the boundary of $X-N$.
\end{remark}





A \index{Symmetric Riemann surface} \textit{symmetric Riemann surface} $(X,\iota)$ is a Riemann surface $X$ with an anti-analytic involution $\iota:X\to X$. For symmetric Riemann surfaces $(X_1,\iota_1)$ and $(X_2,\iota_2)$, a morphism between them is a non-constant continuous morphism $X_1\stackrel{f}{\longrightarrow} X_2$ of Riemann surfaces such that $f\circ\iota_1=\iota_2\circ f$.
\\\\
Given a symmetric Riemann surface $(X,\iota)$, the quotient surface $K=X/\iota$ has a dianalytic structure making the quotient map $\pi:X\to X/\iota$ a morphism of Klein surfaces.
We have $\pi^{-1}(\partial K)=\partial X$ if, and only if, $\pi$ is dianalytic. We call $(X,\iota)$ a \index{Dianalytic symmetric Riemann surfaces} \textit{dianalytic symmetric Riemann surface}.
\\\\
Given a Klein or a symmetric Riemann surface $(X,\iota)$ whose underlying surface has $g$ handles, $0\leq u\leq 2$ crosscaps and $h$ boundary components, we define its topological type as the triple $(g,u,h)$.

\subsection{Nodal Klein and Riemann surfaces} 

Let $(X,N)$ be a singular surface. A \index{Boundary node} \textit{boundary node} is a singularity $z\in N$ with a neighbourhood homeomorphic to a neighbourhood $B\ni(0,0)$, where we define $B:= \{(x,y)\in (\C^+)^2\,|\,xy=0\}$, such that the homeomorphism sends $z$ to $(0,0)$. Similarly, an \index{Interior} \textit{interior node} is a singularity with a neighbourhood homeomorphic to $I\ni(0,0)$, where $I:=\{(x,y)\in \C^2|xy=0\}$. If $X$ has only nodal singularities, then an atlas on $X$ is given by charts on $X-N$ together with charts at the nodes. We call a singular surface with only nodal singularities a \index{Nodal surface} \textit{nodal surface}.
\\\\
A map $f:I\to \C$ is called (anti-)analytic if the compositions $f\circ g$, where $g:\C\to I$ can either send $z$ to $(z,0)$ or $(0,z)$ are (anti-)analytic. A map $f:\C\to I$ is called (anti-)analytic if the composition $\C\stackrel{f}{\longrightarrow}I\hookrightarrow \C^2$ has (anti-)analytic components.
\\

A \index{Nodal Riemann surface} \textit{nodal Riemann surface} is a nodal surface $(X,N)$ together with a maximal analytic atlas. A \index{Nodal Klein surface} \textit{nodal Klein surface} is a nodal surface $(X,N)$ together with a maximal dianalytic atlas. An irreducible component of a nodal surface is a connected component of the surface obtained by pulling apart all the nodes.
A \index{Nodal symmetric Riemann surface} \textit{nodal symmetric Riemann surface} $(X,\iota)$ is a nodal Riemann surface with an anti-analytic involution $\iota:X\to X$. If $\pi(n)$ is a boundary node 
for each node $n\in N$, the surface is called \textit{admissible}.
\\\\
A dianalytic nodal symmetric Riemann surface is an admissible symmetric Riemann surface such that $\pi$ is dianalytic. Observe that this imply that this kind of surface can only have boundary nodes.
\\\\
A Klein surface with $n$ marked points $(X,N)$ is a nodal Klein surface $(X,N)$ with an ordered $n$-tuple $P=(p_1,\dots, p_n)$ of distinct points on $X-N$. A morphism $f:(X_1,P)\to (X_2, P')$ of surfaces with $n$ marked points is a morphism between the underlying surfaces such that $f(p_i)=p_i'$ for each $p_i\in P$ and $p_i'\in P'$.
\\\\
A symmetric Riemann surface $(X,\iota)$ with $(m,n)$ marked points is given by $(X,\iota,P,P')$, where $(X,\iota)$ is a nodal symmetric Riemann surface with an ordered $2m$-tuple of distinct points on $X-N,\, P=(p_1,\dots, p_{2m})$, such that $\iota(p_i)=p_{m+i}$ for $i\in \{1,\dots, m\}$ and an ordered $n$-tuple $P'=(p'_1,\dots, p'_n)$ of distinct points on $X-N$ such that $\iota(p'_j)=p'_j$ for $j\in \{1,\dots, n\}$. A map of marked symmetric Riemann surfaces $f:(X_1,\iota_1, P, P')\to (X_2,\iota_2, Q, Q')$ is a morphism between the underlying symmetric Riemann surfaces such that $f(p_i)=q_i$ and $f(p'_j)=q'_j$. A marked symmetric Riemann surface is called admissible if the underlying symmetric Riemann surface is admissible and the points $\pi(p_i)$ and $\pi(p'_i)$ all lie in the boundary of $X/\iota$.
\newpage 
The category $\dnklein$ has objects Klein surfaces with only boundary nodes 
and marked points on the boundary equipped with a choice of orientation locally on each marked point; its space of arrows is made of dianalytic morphisms.
The category $\dnsymriem$ has objects dianalytic symmetric Riemann surfaces (possibly with boundary) with marked points. The arrows in $\dnsymriem$ are given by analytic maps.


\begin{proposition}[\cite{thbraun12}, Proposition 5.3.11] \label{equivalence}
There exists an equivalence of categories between $\dnklein$ and $\dnsymriem$. 
\end{proposition}

A Klein or Riemann surface with $n$ marked points, possibly oriented, is \index{Stable surface} \textit{stable} if it has only finitely many automorphisms.
\\\\
Let $\overline{\kkk}_{g,u,h,n}$ be the moduli space of stable Klein surfaces in $\dnklein$ with topological type $(g,u,h)$ and $n$ marked points on the boundary. Let us consider the subspace $\kkk_{g,u,h,n}\subset\overline{\kkk}_{g,u,h,n}$ of non-singular Klein surfaces. These moduli spaces are not empty except for the cases:
$$(g,u,h,n)\in \{(0,0,1,0),(0,0,1,1),(0,0,1,2),(0,0,2,0),(0,1,1,0)\}.$$

If we denote by $\ddd_{g,u,h,n}\subset \overline{\kkk}_{g,u,h,n}$ the subspace consisting of those Klein surfaces whose irreducible components are all discs, we have: 

\begin{proposition}[\cite{thbraun12}, Proposition 5.5.9] \label{equiv}
The inclusion $\ddd_{g,u,h,n}\hookrightarrow\overline{\kkk}_{g,u,h,n}$ defines a homotopy equivalence.
\end{proposition}

\section{The definition of an open-closed Klein TCFT}


Let $\Lambda$ be a set whose objects will be called D-branes. We define a topological category $\www_\Lambda$ where:
\begin{enumerate}
\item The class of objects $\ob(\www_\Lambda)$ is given by quadruples $\alpha:= \left([O], [C], s, t\right)$, with $O,\,C\in \mathbb{N}$, where $[O]=\{0,\dots, O-1\}$ and $[C]=\{0,\dots, C-1\}$, and maps $s,t:[O]\to \Lambda$ ;
\item the space of morphisms $\www_\Lambda(\alpha,\beta)$ is given by the moduli spaces of Klein surfaces $\Sigma$ with $\alpha$ incoming boundary components and $\beta$ outgoing boundary components. The closed boundary components are parameterised circles, equipped with an orientation, labelled in $[C]$; the open boundary components are disjoint parameterised intervals, equipped with an orientation, embedded in the remaining boundary components and labelled in $[O]$. An open interval in $\partial \Sigma$ has associated an ordered pair $\{s(i), t(i)\}$ of D-branes indicating where the interval begins and where it ends, respectively. 
Surfaces in $\www_\Lambda(\alpha,\beta)$ have free boundary components, which can be either intervals or circles. 
Free boundary components are the remaining components of $\partial\Sigma$ after removing from it both open and closed components and must be labelled by D-branes in a way compatible with the labelling $\{s(i), t(i)\}$. 
\end{enumerate}
We denote by $\www_{\Lambda,\text{open}}\subset \www_\Lambda$ the full subcategory with objects of the form $\alpha=([O],\emptyset, s,t)$.
\\\\
Composition of morphisms is given by gluing Klein surfaces: we glue together incoming open (resp. closed) boundary components with outgoing open (resp. closed) boundary components. Open boundary components can only be glued together if their D-brane labelling and their orientations agree. Disjoint union makes $\www_\Lambda$ into a symmetric monoidal category.
\\\\
We require the positive boundary condition: Klein surfaces are required to have at least one incoming closed boundary component, or at least one free boundary component, on each connected component.

\begin{remark}
We allow the following exceptional surfaces: the disc, the annulus and the M\"obius strip with no open or closed boundary components and only free boundary components; these surfaces are unstable and so we define their associated moduli space to be a point.
\end{remark}

\begin{figure}[h]
\centering
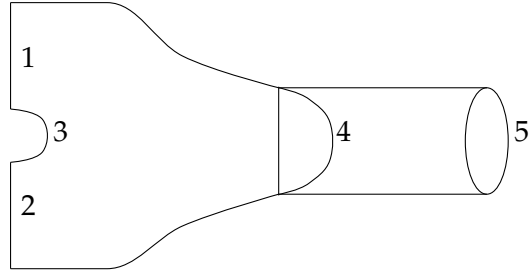
\caption{\textit{Components 1 and 2 are open; components 3 and 4 are free and component 5 is closed.}}
\end{figure}
Let us consider the functor $\CC:\topo\to\comp_{\K}$ of singular chains. The functor $\CC$ yields a 
DGSM category $\oc_\Lambda=\CC(\www_\Lambda)$ whose objects are finite sets $\alpha=([O], [C],s,t)$ and where the space of morphisms is $\homo_{\oc_\Lambda}(\alpha, \beta):=\CC(\www_\Lambda(\alpha, \beta)).$
Let $\oo_\Lambda$ be the full subcategory whose objects are of the form $([O],\emptyset,s,t)$. Similarly, let $\ccc_\Lambda$ be the full subcategory whose objects are of the form $(\emptyset, [C], s,t)$.
\\\\
An \index{Open-closed Klein TCFT} \textit{open-closed Klein topological conformal field theory} (henceforth an open-closed KTCFT) is a pair $(\Lambda,\FF)$ where $\Lambda$ is finite set of D-branes and $\FF$ is a h-split symmetric monoidal functor $\FF:\oc_\Lambda\to \comp_{\K}$; a morphism of open-closed KTCFTs $(\Lambda_1,\FF_1)\to (\Lambda_2, \FF_2)$ is given by a map $\Lambda_1\to\Lambda_2$ and a morphism $\FF\to \LL^\star\FF_2$, where $\LL:\oc_{\Lambda_1}\to \oc_{\Lambda_2}$ is the functor induced by the map $\Lambda_1\to\Lambda_2$; an \index{Open Klein TCFT} \textit{open KTCFT} is a h-split symmetric monoidal functor 
$$\FF:\oo_\Lambda\to \comp_{\K};$$ 
a \index{Closed Klein TCFT} \textit{closed KTCFT} is defined as a h-split symmetric monoidal functor 
$$\FF:\ccc_\Lambda\to \comp_{\K}.$$ 
Morphisms between open (resp. closed) KTCFTs are defined the same way we defined a morphism between open-closed KTCFTs.



\section{Categories via generators and relations}

\subsection{Moduli spaces and categories} 

We define the moduli space $\overline{\kkk}_\Lambda(\alpha,\beta)$ of Klein surfaces in $\dnklein$ (so we allow nodes) as follows: its elements are stable Klein surfaces with $\alpha$ incoming boundary components labelled by $[O_\alpha]$: we assume there are no closed incoming boundary components. Surfaces have $\beta$ outgoing boundary components labelled in a similar way: $O_\beta$ open boundary components and $C_\beta$ closed boundary components labelled by $[O_\beta]$ and $[C_\beta]$ respectively. Closed boundary components have exactly one marked point on them, whilst open marked points are distributed all along the boundary components of the surfaces. Klein surfaces in $\overline{\kkk}_\Lambda(\alpha,\beta)$ have free boundary components, which are the intervals between open marked points and those components with no marked points on them; free boundary components must be labelled by D-branes in $\Lambda$ in a way compatible with the maps $s,t:[O]\to \Lambda$. Let us remark that, although surfaces in $\overline{\kkk}_\Lambda(\alpha,\beta)$ are asked to be stable, we allow the following exceptional surfaces: the disc with zero, one or two open marked points, the annulus with no open or closed points and the M\"obius strip with no open or closed points. Let $\kkk_\Lambda(\alpha,\beta)\subset \overline{\kkk}_\Lambda(\alpha,\beta)$ be the subspace of non-singular Klein surfaces.

\begin{figure}[h]
\centering
\scalebox{3.5}{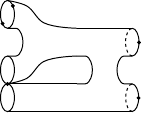}
\caption{\textit{A surface in $\overline{\kkk}_\Lambda(\alpha,\beta)$.}}
\end{figure}

\begin{remark}
Observe that, as we have contracted the intervals, D-branes defining free boundary components are now the intevals between marked points. 
\end{remark}

Let us define $\gcal_\Lambda(\alpha,\beta)\subset \overline{\kkk}_\Lambda(\alpha,\beta)$ as the subspace consisting of Klein surfaces whose irreducible components are either a disc or an annulus of modulus one. Annuli are required to have one of their sides labelled as an outgoing boundary component. Observe that $\gcal_\Lambda(\alpha,\beta)$ contains the exceptional surfaces.


\begin{proposition}\label{inclusion}
The inclusion $\gcal_\Lambda(\alpha,\beta)\hookrightarrow \overline{\kkk}_\Lambda(\alpha,\beta)$ is a weak homotopy equivalence of orbi-spaces.
\end{proposition}

\begin{proof}
This result follows from Proposition \ref{equiv} if one observes that the weak homotopy equivalence $\iota:\ddd_{g,u,h,n}\hookrightarrow \overline{\kkk}_{g,u,h,n}$ holds if we replace points on the interior of each surface in $\ddd_{g,u,h,n}$ and their images by $\iota$ in $\overline{\kkk}_{g,u,h,n}$ with boundary components; we replace at most one point in the same disc. The equivalence holds if we include one marked point in each new boundary component.
\end{proof}

For $\alpha,\beta,\gamma\in\ob(\www_\Lambda)$ with $O_\alpha=O_\beta=O_\gamma=0$, there is a category $\overline{\kkk}_{\Lambda,\text{\scriptsize open}}$ with composition given by maps 
$$\overline{\kkk}_{\Lambda,\text{\scriptsize open}}(\alpha,\beta)\times\overline{\kkk}_{\Lambda,\text{\scriptsize open}}(\beta,\gamma)\to\overline{\kkk}_{\Lambda,\text{\scriptsize open}}(\alpha,\gamma)$$ 
which glue outgoing boundary components in $\overline{\kkk}_{\Lambda,\text{\scriptsize open}}(\alpha,\beta)$ to incoming boundary components in $\overline{\kkk}_{\Lambda,\text{\scriptsize open}}(\beta,\gamma)$. The exceptional surfaces are glued as follows: gluing the disc with two outgoing marked points, one incoming and one outgoing, or both incoming, to a surface $\Sigma$ corresponds to gluing the points of $\Sigma$ together. Gluing the disc with one marked point to a marked point of $\Sigma$ corresponds to forgetting the marked point.
\begin{figure}[h]
\centering
\scalebox{1.25}{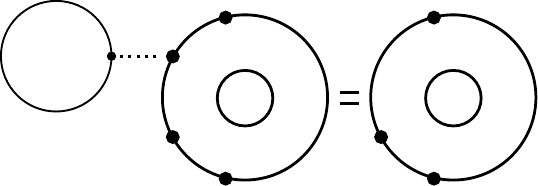}
\end{figure}

The inclusion in Proposition \ref{inclusion} leads to a subcategory $\gcal_{\Lambda, \text{\scriptsize open}}\subset \overline{\kkk}_{\Lambda, \text{\scriptsize open}}$. Observe that disjoint union gives $\overline{\kkk}_{\Lambda, \text{\scriptsize open}}$ and $\gcal_{\Lambda, \text{\scriptsize open}}$ the structure of symmetric monoidal categories. The following result is the unoriented analogue of Proposition 6.1.5 \cite{Costello07}:

\begin{proposition}
The DGSM category $\CC\left(\overline{\kkk}_{\Lambda, \text{\scriptsize open}}\right)$ is quasi-isomorphic to $\oo_\Lambda$. Under the quasi-equivalence 
between $\ob\left(\oc_\Lambda\right)\mhyphen\CC\left(\overline{\kkk}_{\Lambda, \text{\scriptsize open}}\right)$-bimodules and $\ob\left(\oc_\Lambda\right)\mhyphen\oo_\Lambda$-bimodules, $\CC\left(\overline{\kkk}_\Lambda\right)$ is quasi-isomorphic to $\oc_\Lambda$ 
\end{proposition}

\begin{proof}[Sketch of the proof]
The proof for this result is akin to the proof for Proposition 6.1.5 \cite{Costello07}. Let us remind the main points: to start off, the main idea is to find a category $\www'_{\Lambda,\text{\scriptsize open}}$ with the same objects as $\www_{\Lambda,\text{\scriptsize open}}$ together with functors setting homotopy equivalences between the spaces of arrows. Then we will just need to apply $\CC$ and the result will follow. The new category will be created, essentially, by thickening the marked points to transform them into intervals.
\\\\
For a pair of objects $\alpha,\beta\in \ob(\www_\Lambda)$, let us denote $\www'_{\Lambda,\text{\scriptsize open}}(\alpha,\beta)$ the moduli space of Klein surfaces in $\dnklein$ (like $\overline{\kkk}_\Lambda(\alpha,\beta)$) where the marked open boundaries have been replaced by parameterized intervals (like $\www_\Lambda(\alpha,\beta)$). We do not allow these intervals to intersect each other or the nodes on the boundary of the surfaces. By associating each outgoing open boundary interval with a number $t\in[0,1/2]$, we define gluing maps 
$$\www'_{\Lambda,\text{\scriptsize open}}(\alpha,\beta)\times \www'_{\Lambda,\text{\scriptsize open}}(\beta,\gamma)\to \www'_{\Lambda,\text{\scriptsize open}}(\alpha,\gamma),$$ 
making $\www'_{\Lambda,\text{\scriptsize open}}$ into a category. 
\\\\
Inclusions $\www_{\Lambda,\text{\scriptsize open}}(\alpha,\beta)\hookrightarrow \www'_{\Lambda,\text{\scriptsize open}}(\alpha,\beta)$ and $\overline{\kkk}_{\Lambda,\text{\scriptsize open}}(\alpha,\beta)\hookrightarrow \www'_{\Lambda,\text{\scriptsize open}}(\alpha,\beta)$ mapping $0$ and $1/2$ to open boundaries respectively define homotopy equivalences on the spaces of morphisms.
\end{proof}

We follow \cite{Costello07} to give $\gcal_\Lambda(\alpha,\beta)$ a cell decomposition. Assuming that $\K$ has characteristic zero, let $\Sigma\in \gcal_\Lambda(\alpha,\beta)$ and assume $A\subset \Sigma$ is an irreducible component which is an annulus with a closed boundary. We write $A_{\text{\scriptsize open}}$ and $A_{\text{\scriptsize closed}}$ for the open and closed boundary components of $A$, respectively. Let $p\in A_{\text{\scriptsize closed}}$ be the unique marked point in the closed boundary component. We can identify $A$ with the cylinder $\s^1\times [0,1]$ in a way that identifies $p$ with the point $(1,0)$. This identification allows us to cut $A$ from $p$ to a point of $A_{\text{\scriptsize open}}$. We declare:
\begin{enumerate} 
\item the 0-cells are the marked points, the nodes and the intersection points between the cut and $A_{\text{\scriptsize open}}$;
\item the 1-cells are defined to be the boundary components $A_{\text{\scriptsize open}}$, $A_{\text{\scriptsize closed}}$ and the cut;
\item the 2-cell is $\Sigma$.
\end{enumerate}
\begin{figure}[h]
\centering
\scalebox{1}{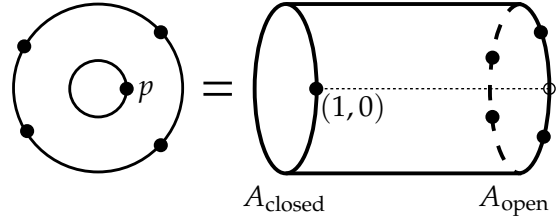}
\caption{In this picture the 0-cells are the marked points, the point (1,0) and the ``white point'' on the right, whereas the 1-cells are $A_{\text{\scriptsize open}}$, $A_{\text{\scriptsize closed}}$ and the ``dotted line''.}
\end{figure}

This process yields a stratification of $\gcal_\Lambda(\alpha,\beta)$ by saying that two surfaces are in the same level if the corresponding marked 2-cell complexes are isomorphic.
\\\\
Let $\CC^{\text{cell}}$ be the functor taking a finite cell complex to an object in $\comp_{\K}$ (see Apendix A \cite{Costello07}); we define the bimodule $\ddd_\Lambda(\alpha,\beta):=\CC^{\text{cell}}\left(\gcal_\Lambda(\alpha,\beta)\right)$
which, by the quasi-isomorphism $\CC^{\text{\scriptsize cell}}(X)\to \CC(X)$ (for $X$ an orbi-cell complex), leads to the following result, which is the unoriented analogue of Lemma 6.1.7 \cite{Costello07}:

\begin{proposition} \label{quasiisos}
There is a quasi-isomorphism of DGSM categories: $\ddd_{\Lambda, \text{open}}\cong \oo_\Lambda$, where we define $\ddd_{\Lambda, \text{open}}(\alpha,\beta)$ as $\CC^{\text{\scriptsize cell}}\left(\gcal_{\Lambda, \text{\scriptsize open}}(\alpha,\beta)\right)$, whereas $\ddd_\Lambda$ is quasi-isomorphic to $\oc_\Lambda$.
\end{proposition}

\subsection{Generators and relations} 

Using the equivalences of categories stated in Proposition \ref{equivalence}, we can move some of the results in \cite{Costello07} into the Klein setting. This implies the definition of several categories, analogous to those appearing in \cite{Costello07}, which will simplify the problem of understanding KTCFTs in terms of involutive $A_\infty$-categories.
\\
A DG category $\aaa$ is generated by some set of arrows $A$ if $\homo_{\aaa}$ has $A$ as a generating set; $\aaa$ has $R$ as a set of relations if $\homo_{\aaa}$ is given by the quotient $A/R$. We say that $\aaa$ is generated as a symmetric monoidal category by $A$ modulo $R$ if $\homo_{\aaa}$ is of the form $A/R$ and the axioms of symmetric monoidal categories are satisfied.
\\\\
Let $\ddd^+_{\Lambda, \text{open}}\subset \ddd_{\Lambda,\text{open}}$ be the subcategory with the same objects but where a morphism is given by a disjoint union of discs, with each connected component having exactly one outgoing boundary marked point. For an ordered set $\lambda_0, \dots, \lambda_{n-1}$ of D-branes, with $n\geq 1$, let $[\lambda_n]:=\{\lambda_0, \dots, \lambda_{n-1}\}\in \ob\left(\oc_\Lambda\right)$ with $O=n$, $s(i)=\lambda_i, t(i)=\lambda_{i+1}$ for $0\leq i\leq O-1$; we use the notation $[\lambda_n]^c:=\{\lambda_1, \dots, \lambda_{n-1},\lambda_0\}$. Let us define $D^+(\lambda_0,\dots,\lambda_{n-1})$ as the disc with $n$ marked points and D-brane labelling given by the different $\lambda_i$, where all the boundary marked points are incoming except for that between $\lambda_{n-1}$ and $\lambda_0$, which is outgoing. The boundary components of the discs are compatibly oriented. There are exceptional morphisms in $\ddd^+_{\Lambda, \text{open}}$ given by discs $D^\tau(\lambda_i,\lambda_{i+1})$ (for $i\in\{0,\dots, n-2\}$), which will be called a \textit{twisted discs}. The particularity of these discs is that, contrary to the discs $D^+(\lambda_0,\dots,\lambda_{n-1})$, they have boundary components oriented incompatibly. 
\vspace{.25cm}
\begin{figure}[h]
\centering
\scalebox{2}{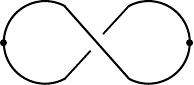}
\caption{\textit{A twisted disc.}}
\end{figure}

Let $\ccc\subset \ddd_{\Lambda,\text{open}}$ be the subcategory with $\ob\left(\ccc\right)=\ob\left(\ddd_{\Lambda,\text{open}}\right)$ but whose arrows are not allowed to have connected components which are the disk with at most 1 open marked point, or the disc with two open marked incoming points or the annulus with neither open nor closed marked points. The morphisms in $\ccc$ are assumed to be not complexes but graded vector spaces.

\begin{proposition}\label{generators}
Let $D(\lambda_0,\dots, \lambda_{n-1})$ be the disc in $\ddd_{\Lambda,\text{open}}$ whose marked points are all incoming. The subcategory $\ccc$ is freely generated, as a symmetric monoidal category over $\ob\left(\ddd_{\Lambda,\text{open}}\right)$, by the discs $D(\lambda_0,\dots, \lambda_{n-1})$ (for $n\geq 3$), the twisted discs $D^\tau(\lambda_i,\lambda_{i+1})$ (for $0\leq i\leq n-2$) and the discs with two outgoing marked points, subject to the relation that $D(\lambda_0,\dots, \lambda_{n-1})$ is cyclically symmetric: $D(\lambda_0,\dots, \lambda_{n-1})=\pm D(\lambda_1,\dots, \lambda_{n-1}, \lambda_0)$.
\end{proposition}

\begin{proof}
The proof for this result follows the steps of Proposition 6.2.1 \cite{Costello07}. If we denote by $\widetilde\eee$ a category with the set of generators and relations stated in the assumptions of the Proposition. We can construct a fully faithful functor $\widetilde{\eee}\to\ccc$, indeed: to prove that the functor is full we observe every surface in $\homo_{\ccc}(\alpha, \beta)$ can be built using disjoint unions of surfaces in $\widetilde{\eee}$ and gluing discs. Observe that the twisted disc, as remarked above, allows us to change the orientations of the marked points, whilst the disc with two outgoing marked points turns incoming boundaries into outgoing boundaries.
\\\\
In order to check that $\widetilde{\eee}\to\ccc$ is faithful, we construct an inverse functor $\ccc\to\widetilde{\eee}$, which is the identity on objects . Let us consider $\Sigma\in\ccc(\alpha,\beta)$, then we can write $\Sigma=\Sigma'\circ \Upsilon$, where both $\Sigma',\Upsilon$ are surfaces in $\widetilde{\eee}$. The surface $\Sigma'$ is composed by disjoint unions of identity maps, discs with all incoming boundaries and twisted discs; the surface $\Upsilon$ is composed by disjoint unions of identity maps, discs with two outgoing boundaries and twisted discs. This decomposition allows us to write a map $\ccc(\alpha,\beta)\to \widetilde{\eee}(\alpha,\beta)$. We conclude that the functor $\widetilde{\eee}\to\ccc$ is faithful.
\end{proof}


\begin{proposition} \label{KleinMess}
The category $\ddd^+_{\Lambda, \text{\scriptsize open}}$ is freely generated as a differential graded symmetric monoidal category 
, by the discs $D^+(\lambda_0,\dots,\lambda_{n-1})$ and $D^\tau(\lambda_i,\lambda_{i+1})$, modulo the relations: 
\begin{enumerate}
\item For $n=2:\, D^\tau(\lambda_0,\lambda_1)\circ D^\tau(\lambda_0,\lambda_1)=\id_{\{\lambda_0,\lambda_1\}}$;
\label{one}
\begin{figure}[h]
\centering
\scalebox{1.5}{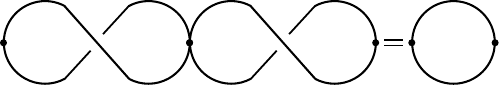}
\end{figure}


\item for $n=3$ we have: $D^+(\lambda_0,\lambda_0,\lambda_1)\circ D^+(\lambda_0)=\id_{\{\lambda_0,\lambda_1\}}=D^+(\lambda_0,\lambda_1,\lambda_1)\circ D^+(\lambda_1);$

\begin{figure}[h]
\centering
\scalebox{1.8}{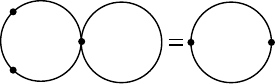}
\end{figure}
\item for $n\geq 3$, gluing twisted discs $D^\tau(\lambda_i,\lambda_{i+1})$ to each incoming boundary component of the discs $D^+(\lambda_0,\dots,\lambda_{n-1})$ is equivalent to gluing a twisted disc $D^\tau(\lambda_0,\lambda_{n-1})$ to the outgoing boundary component of $D^+(\lambda_0,\dots,\lambda_{n-1})$; \label{kleinthree} 
\begin{figure}[H]
\centering
\scalebox{.8}{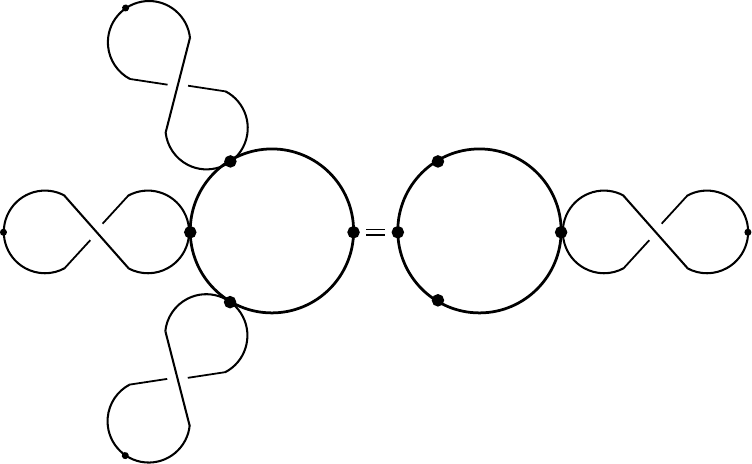}
\end{figure}

\item for $n\geq 4:\, D^+(\lambda_0,\dots,\lambda_i,\lambda_i,\dots, \lambda_{n-1})\circ D^+(\lambda_i)=0$. \label{kleinfour}
\end{enumerate}
\end{proposition}

\begin{remark}
Observe that relation \ref{kleinfour} means that gluing properly a disc with one marked point to $D^+(\lambda_0,\dots,\lambda_i,\lambda_i,\dots, \lambda_{n-1})$ deletes the corresponding marked point. This relation is easy to check and we can have an intuition of its validity when we think of surfaces representing cell complexes.
\end{remark}

\begin{proof}[Sketch of the proof]
This result is the analogue of Lemma 6.2.2 \cite{Costello07}. In particular the relations hold if we depict the surfaces with an appropiate labelling of the boundary marked points.
\end{proof}

\begin{theorem} \label{theorem_Mess}
Let $D_{in}(\lambda_0,\lambda_1)$ and $D_{out}(\lambda_0,\lambda_1)$ be the discs with two incoming or two outgoing boundary components respectively. The category $\ddd_{\Lambda, \text{open}}$ is freely generated, as DGSM category over $\ob(\ddd_{\Lambda,\text{open}})$, by $\ddd^+_{\Lambda, \text{open}}$, $D_{in}(\lambda_0,\lambda_1)$ and $D_{out}(\lambda_0,\lambda_1)$ modulo the following relations:
\begin{enumerate}
  \item An appropiate gluing of a disc with two outgoing boundary components to a disc with two incoming boundary components yields the identity;
  
  


  \item the disc $D(\lambda_0,\dots,\lambda_{n-1})$, whose marked points are all incoming, is cyclically symmetric under the existing permutation isomorphism $[\lambda_n]\cong[\lambda_n]^c.$
\end{enumerate}
\end{theorem}

\begin{remark}
An appropiate gluing of discs means that we have to glue one disc after the other; if we glue their boundary components we would get an annulus, and this is not what we are looking for.
\end{remark}

\begin{proof}
The proof follows the arguments used in Proposition \ref{generators}.
\end{proof}

We denote by $A(\lambda_0,\dots, \lambda_{n-1})$ the annulus with $n\geq 1$ marked points and the intervals between them labelled with D-branes with the inner boundary component labelled as closed. As in the case of the discs $D^+(\lambda_0,\dots,\lambda_{n-1})$, the boundary components of the annuli $A(\lambda_0,\dots, \lambda_{n-1})$ are compatibly oriented. 

\begin{theorem}\label{annuli_Mess}
The annuli $A(\lambda_0,\dots, \lambda_{n-1})$, 
the identity in $\ddd_{\Lambda, \text{open}}(\alpha,\alpha)$ and the twisted discs freely generate $\ddd_\Lambda$ as an $\ob\left(\oc_\Lambda\right)\mhyphen\ddd_{\Lambda, \text{\scriptsize open}}$-bimodule, modulo the following relations: 
\begin{enumerate}
\item Gluing the disc with one marked point $D(\lambda_i)$ to $A(\lambda_0,\dots, \lambda_{n-1})$ in any of the boundary marked points except that between $\lambda_{n-1}$ and $\lambda_0$ yields zero;
\item the disjoint union of the identity element on $\alpha$ with that on $\beta$ is the identity on $\alpha\sqcup \beta$.
\end{enumerate}
\end{theorem}

\begin{proof}
This result follows from Proposition \ref{generators}.
\end{proof}


Let $\ddd^+_\Lambda$ be the $\ob\left(\oc_\Lambda\right)\mhyphen\ddd^+_{\Lambda,\text{\scriptsize open}}$-bimodule with the generators and relations stated above.

\subsection{The differential in $\ddd_\Lambda$} \label{differential}

The definition of the differential for $\mathpzc{D}_\Lambda$ given in \cite{Costello07} can be used in our context. The complexes $\ddd_\Lambda$ admit a differential $d$ which is defined on discs as follows: if $*$ denotes the gluing of the open marked points between $\lambda_i$ and $\lambda_j$:
$$d(D(\lambda_0,\dots,\lambda_{n-1}))=\sum_{\substack{0\leq i \leq j\leq n-1 \\ 2\leq j-i}} \pm D(\lambda_i,\dots, \lambda_j)*D(\lambda_j,\dots, \lambda_i).$$

For annuli, the differential is:
\begin{eqnarray}
d(A(\lambda_0,\dots,\lambda_{n-1})) & = & \sum_{\substack{0\leq i < j\leq n-1 \\ 2\leq |i-j|}} \pm A(\lambda_0,\dots,\lambda_i,\lambda_j,\dots, \lambda_{n-1})*D(\lambda_i,\dots, \lambda_j) \nonumber \\
 & + & \sum_{\substack{0\leq j\leq i\leq n-1 \\ (j,i)\neq (0,n-1)}}\pm A(\lambda_j,\dots, \lambda_i)*D(\lambda_i,\dots, 0,1,\dots, \lambda_j). \nonumber
\end{eqnarray}

\begin{remark}
The signs in the previous formula for the differential are not important for our purposes; nevertheless, we point out that they depend on the orientation chosen for the cells in $\gcal_\Lambda$ of marked points on discs and annuli.
\end{remark}


\begin{lemma}[cf. Lemma 6.3.1 \cite{Costello07}]\label{Costello631}
The assertions below hold:
\begin{enumerate}
\item The $\ob\left(\oc_\Lambda\right)\mhyphen\ddd_{\Lambda,\text{open}}$-bimodule $\ddd_\Lambda$ is $\ddd_{\Lambda,\text{open}}$-flat.
\item If $M$ is a h-split $\ddd_{\Lambda,\text{open}}$-module, then $\ddd_\Lambda\otimes_{\ddd_{\Lambda,\text{open}}}M$ is a h-split $\ob\left(\oc_\Lambda\right)$-module.
\end{enumerate}
These results also hold if one considers $\ddd^+_{\Lambda,\text{open}}$ and $\ddd^+_{\Lambda}$ instead of $\ddd_{\Lambda,\text{open}}$ and $\ddd_{\Lambda}$.
\end{lemma}

\begin{proof}
The $\ob\left(\oc_\Lambda\right)\mhyphen\ddd_{\Lambda,\text{open}}$-bimodule $\ddd_\Lambda$ is generated, for $\alpha\in\ob(\ddd_{\Lambda,\text{open}})$, by the identity elements in $\ddd_{\Lambda,\text{open}}(\alpha,\alpha)$, the twisted discs $D^\tau(\lambda_i,\lambda_{i+1})$ and $A(\lambda_0,\dots, \lambda_{n-1})$.
\\\\
We can filter $\ddd_\Lambda$ as a bimodule with a filtration on the generators by saying that the identity element in $\ddd_\Lambda(\alpha,\alpha)$ and the twisted discs are in $F^0$ and each annulus $A(\lambda_0,\dots,\lambda_{n-1})$ is in $F^n$.
\\\\
In order to show the first point of the Lemma, we have to prove that the functor $\ddd_\Lambda\otimes_{\ddd_{\text{open}}}(-)$ is exact, that is: given a quasi-isomorphism $M_1\to M_2$ of h-split $\ddd_{\Lambda,\text{open}}$-modules we must prove that the map below is also a quasi-isomorphism:
\[
\ddd_\Lambda(\beta,-)\otimes_{\ddd_{\Lambda, \text{open}}}M_1(-) \to \ddd_\Lambda(\beta,-)\otimes_{\ddd_{\Lambda, \text{open}}}M_2(-).
\]
Giving both sides the filtration induced by $\ddd_\Lambda(\beta, -)$, it is enough to show the statement on the associated graded complexes.
\\\\
Let $\alpha\in\ob(\ddd_{\Lambda,\text{open}})$ and $\ob\left(\oc_\Lambda\right)\ni\beta=C\sqcup \alpha$ for $C\in \mathbb{N}$; observe that we are adding $C$ closed states to $\alpha$. We will show the result for $C=1$. Let $M$ be a h-split $\ddd_{\Lambda,\text{open}}$-module. In degree $n$, by sending the generators of $\ddd_\Lambda$ which are the identity in $\ddd_{\Lambda,\text{open}}$ to $\alpha$ and the annulus $A(\lambda_0,\dots, \lambda_{n-1}):=\mathpzc{a}$ to $[\lambda_n]^c$, we get that
$\ddd_\Lambda(\alpha\sqcup 1,-)\otimes_{\ddd_{\Lambda, \text{open}}}M(-)$ is spanned by the spaces $\mathpzc{a}\otimes_{\K}M(\alpha\sqcup[\lambda_n]^c)$.
\\\\
Let us introduce the following notation: $\widehat{\left[\lambda_n\right]^c}_i:=\{\lambda_1,\dots, \lambda_{i-1},\lambda_{i+1},\dots, \lambda_{n-1}, \lambda_0\}$ and we write $\widecheck{\left[\lambda_n\right]^c}_i:=\{\lambda_1,\dots, \lambda_{i-1},\lambda_i,\lambda_i, \lambda_{i+1},\dots, \lambda_{n-1},\lambda_0\}$. There is just one relation to be considered in $\ddd_\Lambda$: gluing a disc with one boundary marked point to any of the marked points of $\mathpzc{a}$, but that between $\lambda_{n-1}$ and $\lambda_0$, is zero. This is the same as saying that the following composition is zero:
\begin{multline*}
\mathpzc{a}\otimes_{\K} M\left(\alpha\sqcup\widehat{\left[\lambda_n\right]^c}_i\right) \stackrel{(1)}{\longrightarrow}
\mathpzc{a}\otimes_{\K} M\left(\alpha\sqcup\widecheck{\left[\lambda_n\right]^c}_i\right)\stackrel{(2)}{\longrightarrow}\\
\deg^n(\ddd_\Lambda(\alpha\sqcup 1,-)\otimes_{\ddd_{\Lambda,\text{open}}} M(-))
\end{multline*}

The map (1) corresponds to the element $\ddd_{\Lambda,\text{open}}\left(\alpha\sqcup\widehat{\left[\lambda_n\right]^c}_i, \widecheck{\left[\lambda_n\right]^c}_i\right)$, obtained from the tensor product of $\id_{\alpha}$ and $\id_{\widehat{\left[\lambda_n\right]^c}_i}$ with the map corresponding to the disc with a single marked point. Observe that the map (2) corresponds to gluing the disc with one marked point, due to the fact that $\ddd_\Lambda(\alpha\sqcup 1,-)\otimes_{\ddd_{\Lambda,\text{open}}} M(-)$ is spanned by $\mathpzc{a}\otimes_{\K}M(\alpha\sqcup[\lambda_n]^c)$. 

\begin{figure}[h]
\centering
\scalebox{.89}{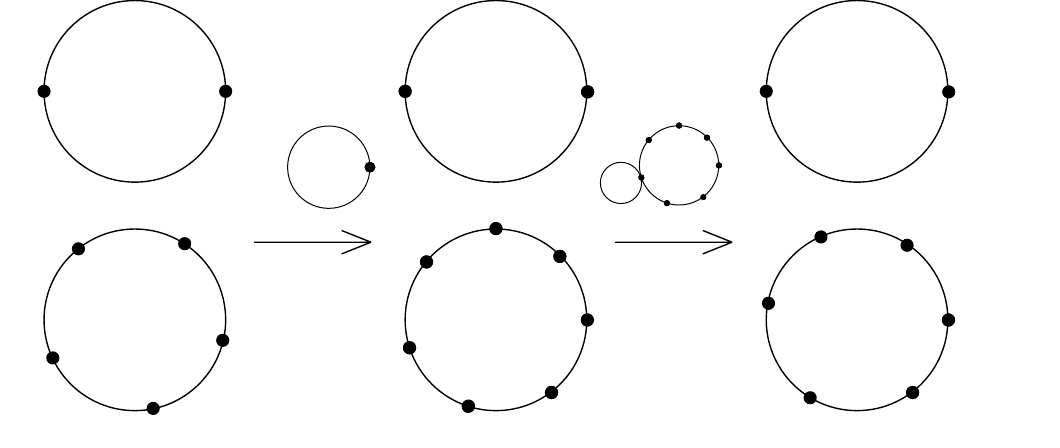}
\end{figure}
\newpage
As (1) is always injective (because we can find an splitting coming from the disc with one marked point), taking quotient is an exact operation and hence 
$$\ddd_\Lambda(\alpha\sqcup 1,-)\otimes_{\ddd_{\Lambda,\text{open}}} M(-)$$ 
is an exact functor. The same argument applies for any $C\in \mathbb{Z}$ by observing that, as each annulus $A(\lambda_0,\dots,\lambda_{n-1})$ has a closed boundary component, each integer $C$ corresponds to an annulus, which \textit{contributes} with an element of the form $\left[[\lambda]^c_n\right]$. Therefore the first part of the Lemma is proved. 
\\\\
The second part is proved similarly. Let $N:=\ddd_{\Lambda}\otimes_{\ddd_{\Lambda,\text{open}}} M$ and, for simplicity, assume $\beta=\alpha\sqcup 1$. In order to show that $N(\beta)\otimes N(\beta')\to N(\beta\sqcup \beta')$ is a quasi-isomorphism, we take the filtration induced by $\ddd_\Lambda$ and check the result for the associated graded complexes. Roughly speaking: for $[\lambda_n],\, [\lambda'_m]\in\ob(\ddd_{\Lambda,\text{open}})$, $N(\beta)$ is spanned by $\mathpzc{a}\otimes_{\K} M(\alpha\sqcup[\lambda_n]^c)$ and $N(\beta')$ is spanned by $\mathpzc{a'}\otimes_{\K} M(\alpha'\sqcup[\lambda'_m]^c)$. Therefore, the tensor product $N(\beta)\otimes N(\beta')$ is spanned by 
$$\mathpzc{a}\otimes_{\K}\mathpzc{a'}\otimes_{\K}M(\alpha\sqcup[\lambda_n]^c)\otimes_{\K} M(\alpha'\sqcup[\lambda'_m]^c)$$
which is quasi-isomorphic to $\mathpzc{a}\otimes_{\K}\mathpzc{a'}\otimes_{\K}M\left((\alpha\sqcup[\lambda_n]^c)\sqcup (\alpha'\sqcup[\lambda'_m]^c)\right)$
as $M$ is h-split. These elements span $\ddd_\Lambda(\beta\sqcup\beta',-)\otimes_{\ddd_{\Lambda,\text{open}}}M(-)$. We conclude by observing that the same proof works for $\ddd^+_{\Lambda,\text{open}}$ and $\ddd^+_{\Lambda}$.
\end{proof}




\section{Calabi-Yau involutive categories and KTCFTs}

\subsection{Calabi-Yau involutive $A_\infty$-categories} 

\begin{remark}
Henceforth, we will assume that the field $\K$ has characteristic zero and that it is equipped with an involution given by the identity map.
\end{remark}



An \index{Involutive $A_\infty$-category} \textit{involutive $A_\infty$-category} $\mathpzc{C}$ consists of:
\begin{enumerate}
  \item A class of objects $\ob(\mathpzc{C})$;
  \item for all $c_1,c_2\in\ob(\mathpzc{C})$, a $\mathbb{Z}$-graded abelian group of morphisms $\homo_{\mathpzc{C}}(c_1,c_2)$;
    \item a functor $\star:\mathpzc{C}^\text{op}\to\mathpzc{C}$ which is the identity on objects and satisfying, for morphisms $f,g\in\homo_\mathpzc{C}$: $(f\circ g)^\star=g^\star\circ f^\star$, $(f^\star)^\star=f$ and $\id^\star=\id$;
  \item for all $n\geq 1$, composition maps 
  $$m_n:\homo_{\mathpzc{C}}(c_1,c_2)\otimes\cdots\otimes \homo_{\mathpzc{C}}(c_n,c_{n+1})\to \homo_{\mathpzc{C}}(c_1,c_{n+1})$$ of degree $n-2$ satisfying $\forall n\geq 1$:
  \begin{equation} \label{infinity2}
\sum_{i+j+l=n}(-1)^{i+jl}m_{i+1+l}\circ(\id^{\otimes i}\otimes m_j\otimes \id^{\otimes l})=0; 
\end{equation}
\item Given morphisms $f_1,\dots, f_n\in\homo_\mathpzc{C}$, the maps $m_n$ are required to satisfy the following identity:
$$(m_n(f_1\otimes\cdots\otimes f_{n}))^\star=m_n(f_n^\star\otimes\cdots\otimes f_1^\star).$$
\end{enumerate}

A \index{CY involutive $A_\infty$-category} \textit{Calabi-Yau involutive $A_\infty$-category} is an involutive $A_\infty$-category $\eee$ endowed with a trace map $\tr:\homo_{\mathpzc{C}}(e_1,e_1)\to \K$, satisfying $(\tr(f))^\star=\tr(f^\star)=\tr(f)$, and a symmetric and non-degenerate on homology pairing of chain complexes 
\[
\left.\begin{array}{cccc}\langle - , -\rangle_{e_1,e_2}: & \homo_{\mathpzc{C}}(e_1,e_2)\otimes \homo_{\mathpzc{C}}(e_2,e_1) & \to & \K \\ & f\otimes g & \mapsto &\tr(g\circ f)\end{array}\right.
\]
for each $e_1,e_2\in\ob(\eee)$. This pairing is required to satisfy, for maps $f\in\homo_{\mathpzc{C}}(e_1,e_2)$ and $g\in\homo_{\mathpzc{C}}(e_2,e_1)$:
\begin{equation}\label{twistedinvolutive}
\langle f, g \rangle_{e_1,e_2}= \langle g^\star, f^\star \rangle_{e_1,e_2}
\end{equation}
and the following identity:
$$
\langle m_{n-1}(e_0\otimes \cdots \otimes e_{n-2}), e_{n-1}\rangle=
(-1)^{(n+1)+|e_0|\sum_{i=1}^{n-1}|e_i|}\langle m_{n-1}(e_1\otimes \cdots \otimes e_{n-1}), e_0 \rangle.
$$

Given two Calabi-Yau involutive $A_\infty$-categories $(\aaa,\star)$ and $(\bbb,\dagger)$, a functor $\FF:(\aaa,\star)\to(\bbb,\dag)$ is a functor of the underlying $A_\infty$-categories (Section 5.1.2 \cite{Lefevre03}) such that $\FF\circ\star=\dag\circ\FF$. Calabi-Yau $A_\infty$-categories and functors betweem them form a category.

\subsection{Open KTCFTs and Calabi-Yau involutive $A_\infty$-categories} 

The main result of this chapter states that the category of open KTCFTs is quasi-isomorphic to the category of Calabi-Yau $A_\infty$-categories endowed with involution. We will get products $m_n$ from the generators of the categories defined in the previous sections and, by using the twisted disc $D^\tau(\lambda_0,\lambda_1)$, we will equip all our $A_\infty$-categories with an involution.
\\\\
Let $\FF:\ddd^+_{\Lambda, \text{open}}\to \comp_{\K}$ be a split symmetric monoidal functor. For each $O\in \mathbb{N}$ and D-brane labelling given by $\{s(i), t(i)\}$, with $0\leq i\leq O-1$, the following isomorphism holds:
\begin{equation} \label{isom}
\FF([O],s,t)\cong \bigotimes_{i=0}^{O-1}\FF(\{s(i), t(i)\}).
\end{equation}

Let the pair $\{s(i), t(i)\}$ correspond to the pair of D-branes $\{\lambda_i,\lambda_{i+1}\}$. We can define a category $\bbb$ with $\ob(\bbb):=\Lambda$ and $\homo_{\bbb}(\lambda_i,\lambda_{i+1}):=\FF(\{s(i), t(i)\})$. Composition of morphisms in $\bbb$ makes sense as $\FF$ is split. 
Observe that we are just associating each open boundary component (i.e. each interval and later on each open marked point) to the space $\homo_{\bbb}(\lambda_i, \lambda_{i+1})$. 
\\\\
A \index{CY UEI $A_\infty$-category} \textit{Calabi-Yau unital extended involutive $A_\infty$-category} with objects in $\Lambda$ is a h-split symmetric monoidal functor $\FF:\ddd_{\Lambda, \text{open}}\to\comp_{\K}$. By considering $\ddd^+_{\Lambda,\text{open}}$ instead of $\ddd_{\Lambda,\text{open}}$ we get the concept of \index{UEI $A_\infty$-category}\textit{unital extended involutive $A_\infty$-category}. If we consider split functors instead of h-split functors, we obtain the concept of \index{Unital CY $A_\infty$-category}\textit{unital Calabi-Yau involutive $A_\infty$-category} and the concept of \index{Unital involutive $A_\infty$-category}\textit{unital involutive $A_\infty$-category} respectively. These definitions make sense due to the following Lemmata:

\begin{lemma} \label{invinfinity}
A split symmetric monoidal functor $\FF:\ddd^+_{\Lambda, \text{open}}\to \comp_{\K}$ is the same as a unital involutive $A_\infty$-category $\bbb$ with set of objects $\Lambda$.
\end{lemma}

\begin{proof}
The proof follows from the isomorphism (\ref{isom}) above. Let us observe that:
\begin{enumerate}
\item The twisted disc $D^\tau(\lambda_i,\lambda_{i+1})$ yields the involution 
$$\star:\homo_{\bbb}(\lambda_i,\lambda_{i+1})\to\homo_{\bbb}(\lambda_{i+1},\lambda_i);$$
\item the discs $D^+(\lambda_0,\dots, \lambda_{n-1})$ yield the products
$$m_{n-1}:\homo_{\bbb}(\lambda_0,\lambda_1)\otimes\dots\otimes\homo_{\bbb}(\lambda_{n-2},\lambda_{n-1})\to\homo_{\bbb}(\lambda_0,\lambda_{n-1});$$
\item the differential $d$ gives the $A_\infty$-relations between the $m_n$; 
\item for $n=2,\, D^+(\lambda_0,\lambda_1)$ yields the identity $\homo_{\bbb}(\lambda_0,\lambda_1)\to\homo_{\bbb}(\lambda_0,\lambda_1)$;
\item for $n=1,\, D^+(\lambda)$ yields the unit $\K\to\homo_{\bbb}(\lambda,\lambda)$.
\end{enumerate}
Observe that relation \ref{kleinthree} in Corollary \ref{KleinMess} proves that the products $m_n$ preserve the involution. \qedhere
\end{proof}

\begin{lemma} \label{invcalabi}
A split symmetric monoidal functor $\FF:\ddd_{\Lambda, \text{open}}\to \comp_{\K}$ is the same as a unital Calabi-Yau involutive $A_\infty$-category $\bbb$ with set of objects $\Lambda$.
\end{lemma}

\begin{proof}
The proof follows the same arguments of Lemma \ref{invinfinity} but now we have two more generators (see Theorem \ref{theorem_Mess}): the discs with two incoming and two outgoing marked points, which yield the map
$\homo_{\bbb}(\lambda_0,\lambda_1)\otimes\homo_{\bbb}(\lambda_1,\lambda_0)\to \K$
and its inverse. The extra relations on $\ddd_{\Lambda, \text{open}}$ correspond to the cyclic symmetry condition. As in the previous result, the anti-analytic involution on the Riemann surfaces is transferred to the Calabi-Yau involutive $A_\infty$-category through a twisted disc. Observe that we can deduce the identity $\langle f,g \rangle=\langle g^\star, f^\star\rangle$ from relation \ref{kleinthree} in Corollary \ref{KleinMess}. 
\end{proof}

The following result is clear from the above results, and almost proves the first part of our main theorem:

\begin{proposition}
The category of Calabi-Yau unital extended involutive $A_\infty$-categories with set of objects $\Lambda$ is quasi-equivalent to the category of open KTCFTs.
\end{proposition}

\begin{proof}
Let us recall that an open KTCFT is an h-split monoidal functor 
$$\KK:\oo_{\Lambda}\to \comp_{\K}.$$
The result follows from Lemma \ref{invcalabi} and the quasi-isomorphism between $\oo_\Lambda$ and $\ddd_{\Lambda, \text{open}}$, in Proposition \ref{quasiisos}.
\end{proof}
\begin{proposition}\label{quasi-equivalence}
The following categories, each one with set of objects $\Lambda$, are quasi-equivalent:
\begin{enumerate}
\item The category of unital extended involutive $A_\infty$-categories;
\item the category of unital involutive $A_\infty$-categories and
\item the category of unital involutive DG categories.
\end{enumerate}
\end{proposition}

\begin{proof}
Let $\alpha,\beta\in\ob\left(\ddd^+_{\Lambda,\text{open}}\right)$. The space $\ddd^+_{\Lambda,\text{open}}(\alpha,\beta)$ is contractible as it is given by the chains on the moduli spaces of discs with $\alpha$ incoming marked points and $\beta$ outgoing marked points, hence for $n\neq 0$ we have: 
$$\h_n\left(\ddd^+_{\Lambda,\text{open}}(\alpha,\beta)\right)=0.$$ 
This implies that $\ddd^+_{\Lambda,\text{open}}(\alpha,\beta)$ is quasi-isomorphic to its homology and in particular it is quasi-isomorphic to $\h_0\left(\ddd^+_{\Lambda,\text{open}}(\alpha,\beta)\right)$.
\\\\
With the notation introduced at the beginning of this section, giving a split functor 
$$\FF:\h_0\left(\ddd^+_{\Lambda,\text{open}}(\alpha,\beta)\right)\to \comp_{\K}$$ 
is the same as giving a unital DG category $\bbb$ with set of objects $\Lambda$. Observe that 
$$\h_0\left(\ddd^+_{\Lambda,\text{open}}([\lambda_{n+1}],\{\lambda_0, \lambda_n\})\right)$$ 
corresponds to an ``alien pair of pants'' given by a disc with $n$ marked points with the point between $\lambda_n$ and $\lambda_0$ is outgoing. This corresponds to the product
$$\homo_{\bbb}(\lambda_0,\lambda_1)\otimes\dots\otimes\homo_{\bbb}(\lambda_{n-1},\lambda_n)\to\homo_{\bbb}(\lambda_0,\lambda_n),$$
which is associative as $\h_0\left(\ddd^+_{\Lambda,\text{open}}([\lambda_{n+1}],\{\lambda_0, \lambda_n\})\right)$ has dimension one.
\\\\
We show that there is a quasi-equivalence between unital extended involutive $A_\infty$-categories and unital extended involutive DG categories. For that purpose we will use the equivalences obtained in Lemmata \ref{invinfinity} and \ref{invcalabi}. We define a unital extended involutive DG category as a h-split functor of the form $\GG:\h_0\left(\ddd^+_{\Lambda,\text{open}}\right)\to\comp_{\K}$. Due to the quasi-isomorphism 
$$\ddd^+_{\Lambda,\text{open}}\cong \h_0\left(\ddd^+_{\Lambda,\text{open}}\right)$$ 
there is a quasi-equivalence between unital extended involutive $A_\infty$-categories and unital extended involutive DG categories given by
\begin{equation} \label{invisoopen}
\xymatrix{
\ddd^+_{\Lambda, \text{open}} \ar[rr]^-{\cong} \ar[dr] &  & \h_0\left(\ddd^+_{\Lambda, \text{open}}\right) \ar[dl]^-{\GG} \\
 & \comp_{\K} & 
}
\end{equation}

Our next step is to show that there exist a quasi-equivalence between unital extended involutive $A_\infty$-categories and involutive $A_\infty$-categories. It goes as follows:
\\\\
In $\ddd^+_{\Lambda, \text{open}}$ the following isomorphism holds:

\begin{equation} \label{isodopen}
\bigotimes_{i=1}^n\ddd^+_{\Lambda,\text{open}}(\alpha_i, \{\lambda_i, \lambda'_i\}) \to \ddd^+_{\Lambda,\text{open}}(\sqcup_{i=1}^n\alpha_i, \sqcup_{i=1}^n\{\lambda_i, \lambda'_i\}).
\end{equation}

Let us consider a unital extended involutive $A_\infty$-category given by a h-split functor 
$$\FF:\ddd^+_{\Lambda,\text{open}}\to \comp_{\K}$$ 
and define a unital involutive $A_\infty$-category as a split functor $F_{\FF}:\ddd^+_{\Lambda, \text{open}}\to \comp_{\K}$ by stating:
$$F_{\FF}([O],s,t):=\bigotimes_{i=0}^{O-1}\FF(\{s(i), t(i)\}).$$

This definition together with the isomorphism (\ref{isodopen}) secures the existence of maps $F_{\FF}(\alpha)\to \FF(\alpha)$ which, composing with the action of $\ddd^+_{\Lambda, \text{open}}$ yields maps 
$$F_{\FF}(\alpha)\otimes \ddd^+_{\Lambda,\text{open}}(\alpha,\{\lambda_0,\lambda_1\})\to F_{\FF}(\{\lambda_0,\lambda_1\}),$$ indeed: 
\[
\xymatrix{
F_{\FF}(\alpha)\to \FF(\alpha) \ar[d]\\
F_{\FF}(\alpha)\otimes \ddd^+_{\Lambda, \text{open}}(\alpha, \{\lambda_0,\lambda_1\})\to \FF(\alpha)\otimes \ddd^+_{\Lambda, \text{open}}(\alpha, \{\lambda_0,\lambda_1\}) \ar[d] \\
F_{\FF}(\alpha)\otimes \ddd^+_{\Lambda, \text{open}}(\alpha, \{\lambda_0,\lambda_1\})\to \FF(\{\lambda_0,\lambda_1\})=F_{\FF}(\{\lambda_0,\lambda_1\}) 
}
\]

Due to the isomorphisms (\ref{isodopen}) we get that $F_{\FF}$ is monoidal, what leads us to conclude that $F_{\FF}$ is a $\ddd^+_{\Lambda,\text{open}}$-module.
\\\\
This concludes the proof of the equivalence $(1)\Leftrightarrow (2)$ . Similarly we prove that unital extended involutive DG categories are quasi-equivalent to unital involutive DG categories. Let UEI stand for ``unital extended involutive'', the diagram (\ref{invisoopen}) connects the latter quasi-equivalences in the sense below:
\[
\xymatrix{
\text{UEI $A_\infty$-categories} \ar[r]_-{\simeq}^-{(\ref{invisoopen})} \ar[d]_-{\simeq} & \text{UEI DG categories} \ar[d]^-{\simeq} \\
\text{Unital $A_\infty$-categories} \ar[r]_{\simeq} & \text{Unital DG categories} 
}
\]

The quasi-equivalence $(3)\Leftrightarrow (1)$ is straightforward from the diagram above.
\end{proof}

This concludes the proof of part (1) of Theorem \ref{thetheorem}. Observe that, as we have shown that there are quasi-isomorphisms $\ddd_{\Lambda,\text{open}}\cong \oo_{\Lambda}$ and $\ddd_{\Lambda}\cong \oc_{\Lambda}$ (this is Proposition \ref{quasiisos}), by Proposition \ref{costello444} we have, for a left $\ddd_{\Lambda,\text{open}}$-module $\MM_1$ and its associated left $\oo_{\Lambda}$-module $\MM_2$:
$$\underbrace{\oc_{\Lambda}(-,\beta)\otimes_{\oo_{\Lambda}}^{\mathbb{L}}\MM_2}_{\NN(\beta)}\cong \ddd_{\Lambda}\otimes_{\ddd_{\Lambda,\text{open}}}\MM_1.$$
This shows that, if $\MM_2$ is h-split, so it is $\NN(\beta)$. Therefore $\NN$ defines nothing but an open-closed Klein topological conformal field theory, which is the universal open-closed KTCFT associated to $\MM_2$. This proves part (2) of Theorem \ref{thetheorem}. Our next objective is to prove part (3), concluding the proof of Theorem \ref{thetheorem}; this is the purpose of the next section. 


\section{Open-closed KTCFTs and involutive Hochschild homology}

For an involutive DG category $\aaa$, we define its \index{Involutive Hochschild chain complex} \textit{involutive Hochschild chain complex} as
$$C^{\text{inv}}_\bullet(\aaa)=\bigoplus_n\left( \bigoplus_{a_0,\dots,a_{n-1}}\homo_{\aaa}(a_0,a_1)\otimes \dots \otimes \homo_{\aaa}(a_{n-1},a_0) \right)[1-n]\Biggm/\sim,$$
where $\sim$ denotes the relation $f_0^\star\otimes g=f_0\otimes g^\star$, with $g=(f_1,\dots, f_{n-1})$. 
The involution is given by: $(f_0\otimes\dots\otimes f_{n-1})^\star=f_{n-1}^\star\otimes\dots\otimes f_0^\star$.
\\\\
The differential for $C^{\text{inv}}_\bullet(\aaa)$ is given, for maps $f_i\in\homo_{\aaa}(\alpha_i,\alpha_{i+1})$ by:
\begin{align*}
d(f_0\otimes\dots\otimes f_{n-1}) & = \sum_{i=0}^{n-1}(-1)^i(f_0\otimes\dots\otimes df_i\otimes\dots\otimes f_{n-1}) \nonumber \\
 & + \sum_{i=0}^{n-2}(-1)^i(f_0\otimes\dots\otimes (f_{i+1}\circ f_i)\otimes\dots\otimes f_{n-1}) \nonumber \\
 & + (-1)^{n-1} ((f_0\circ f_{n-1})\otimes\dots\otimes f_{n-2}). \nonumber
\end{align*}

\begin{lemma}
The differential $d$ preserves involutions.
\end{lemma}

\begin{proof}
It is a direct computation:
\begin{align*}
d(f_0\otimes\dots\otimes f_{n-1})^\star & = \sum_{i=0}^{n-1}(-1)^i(f_0\otimes\dots\otimes df_i\otimes\dots\otimes f_{n-1})^\star \\
 & +  \sum_{i=0}^{n-2}(-1)^i(f_0\otimes\dots\otimes (f_{i+1}\circ f_i)\otimes\dots\otimes f_{n-1})^\star \\
 & +  (-1)^{n-1} ((f_0\circ f_{n-1})\otimes\dots\otimes f_{n-2})^\star \\
 & =  d(f^\star_{n-1}\otimes \dots \otimes f^\star_0). \qedhere 
\end{align*}
\end{proof}

If $\aaa$ is unital, the \index{Normalized involutive Hochschild chain complex} \textit{normalized involutive Hochschild chain complex} $\overline{C}^{\text{inv}}_\bullet(\aaa)$ is the quotient of $C^{\text{inv}}_n(\aaa)$ by the sub-complex spanned by $f_0\otimes\dots\otimes f_{n-1}$, where at least one of the maps $f_i$ (for $i>0$) is the identity. We have the following result:

\begin{lemma}[cf. Lemma 7.4.1 \cite{Costello07}]
The functor $\aaa\to\overline{C}^{\text{inv}}_\bullet(\aaa)$, from the category of involutive DG categories with set of objects $\Lambda$ to the category of complexes, is exact.
\end{lemma}


Given an Calabi-Yau extended involutive $A_\infty$-category $\phi$ there is an underlying extended involutive $A_\infty$-category given by restricting to $\ddd^+_{\Lambda,\text{open}}$, indeed: let us consider $\phi:\ddd_{\Lambda,\text{open}}\to \comp_{\K}$ an Calabi-Yau extended involutive $A_\infty$-category (see Lemma \ref{invcalabi}); if we consider the subcategory $\ddd^+_{\Lambda,\text{open}}\subset \ddd_{\Lambda,\text{open}}$ and take the restriction $\phi_{|\ddd^+_{\Lambda,\text{open}}}$, we get a functor $\ddd^+_{\Lambda,\text{open}}\to\comp_{\K}$ which, by Lemma \ref{invcalabi}, is an Calabi-Yau extended involutive $A_\infty$-category; this is the underlying category we are talking about. The Hochschild homology of $\phi$ is defined to be the homology of the associated underlying $A_\infty$-category.

\begin{proposition}
For a unital Calabi-Yau extended involutive $A_\infty$-category $\phi$ the following equality holds:
$$\ddd_\Lambda(-,1)\otimes_{\ddd_{\Lambda,\text{open}}}\phi=\overline{C}^{\text{inv}}_\bullet(\phi).$$ 
\end{proposition}

\begin{proof}
The proof for this result is based on the generators and relations stated in Lemma \ref{KleinMess}, Theorem \ref{theorem_Mess} and Lemma \ref{annuli_Mess}. We can state the following equality, as we have defined the generators of $\ddd^+_\Lambda$ as those of $\ddd_\Lambda$:
$$\ddd_\Lambda(-,1)\otimes_{\ddd_{\Lambda,\text{open}}}\phi=\ddd^+_\Lambda(-,1)\otimes_{\ddd^+_{\Lambda,\text{open}}}\phi.$$
From Lemma \ref{Costello631} we know that the functor $\phi \rightsquigarrow \ddd^+_\Lambda(-,1)\otimes_{\ddd^+_{\Lambda,\text{open}}}\phi$ is exact, so we only have to check that the equality below holds for the DG category $\bbb$ associated to $\phi$, thought of as a left $\ddd^+_{\Lambda,\text{open}}$-module:
$$\ddd^+_\Lambda(-,1)\otimes_{\ddd^+_{\Lambda,\text{open}}}\bbb=\overline{C}^{\text{inv}}_\bullet(\bbb).$$

\begin{remark}
The association between $\phi:\ddd^+_{\Lambda,\text{open}}\to \comp_{\K}$ and the involutive DG category $\bbb$ follows from the quasi-equivalence stated in Proposition \ref{quasi-equivalence}.
\end{remark}

We proceed as in Lemma \ref{Costello631}: in degree $n$, 
the associated complex to the $\ob(\oc_\Lambda)$-module $\ddd^+_\Lambda(-,1)\otimes_{\ddd^+_{\Lambda,\text{open}}}\bbb(-)$ is spanned by $A(\lambda_0,\dots, \lambda_{n-1})\otimes_{\K}\bbb([\lambda_n]^c)$, which is associated, through the disc $D(\lambda_0, \dots, \lambda_{n-1})$, to the product 
$$\homo_{\bbb}(\lambda_0,\lambda_1)\otimes\dots\otimes\homo_{\bbb}(\lambda_{n-1},\lambda_0),$$ 
modulo the subspace spanned by the elements of the form $\phi_0\otimes \dots \otimes \phi_{n-1}$, where at least one of the $\phi_i$ (for $i>0$) is the identity. This quotient comes from the construction of the tensor product $\ddd^+_\Lambda(-,1)\otimes_{\ddd^+_{\Lambda,\text{open}}}\bbb$, indeed: let us recall that the tensor product is characterized by the following commutative diagram:
\[
\xymatrix{
\ddd^+_\Lambda(m,1)\otimes_{\K} \ddd^+_{\Lambda,\text{open}}(n,m)\otimes_{\K} \bbb(n) \ar[r]^-{(1)} \ar[d]_-{(2)}& \ddd^+_\Lambda(m,1)\otimes_{\K} \bbb(m) \ar[d]\\
\ddd^+_\Lambda(n,1)\otimes_{\K} \bbb(n) \ar[r] & \ddd^+_\Lambda(-,1)\otimes_{\ddd^+_{\Lambda,\text{open}}}\bbb(-)
}
\]
\begin{remark}
Mind the abuse of notation: we write $\bbb(m)$ instead of $\bbb([\lambda_m]^c)$.
\end{remark}

Action (1) corresponds to gluing the surface in $\ddd^+_{\Lambda,\text{open}}(n,m)$ to a disc depicting $\bbb(n)$ whilst action (2) corresponds to gluing the same surface in $\ddd^+_{\Lambda,\text{open}}(n,m)$ to an annulus representing $\ddd^+_\Lambda(m,1)$. Algebraically, action (1) corresponds to the map:
\begin{multline*}
\homo_{\bbb}(\lambda_0,\lambda_{1})\otimes \dots \otimes \homo_{\bbb}(\lambda_{m-1},\lambda_0) \to \\ 
\homo_{\bbb}(\lambda_0,\lambda_1)\otimes \dots \otimes \homo_{\bbb}(\lambda_i, \lambda_i)\otimes \dots \otimes \homo_{\bbb}(\lambda_{m-1},\lambda_0)
\end{multline*}

defined by: $f_0\otimes \dots \otimes f_{m-1}\mapsto f_0\otimes \dots \otimes \id_{\{\lambda_i,\lambda_i\}} \otimes \dots \otimes f_{m-1}$. On the other hand, action (2) is zero as stated in Theorem \ref{annuli_Mess}. Keeping in mind that the diagram commutes, we get the relation that the tensor product of maps where at least one is the identity (for $i>0$) yields zero, and this is what leads to the quotient space above. The following picture intends to make this reasoning clearer:

\begin{figure}[h]
\centering
\scalebox{1}{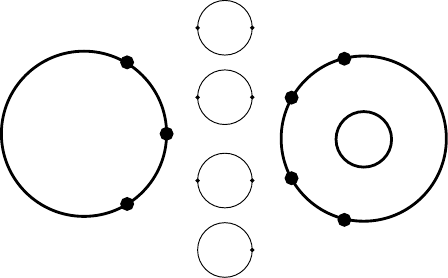}
\caption{\textit{Gluing on the left or on the right must be equivalent}.}
\end{figure}

There is a further relation given by:
\newpage
\begin{figure}[h]
\centering
\scalebox{1}{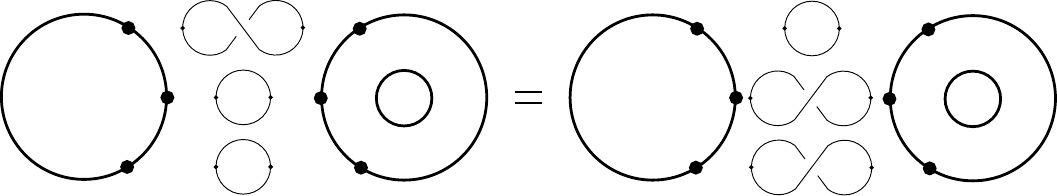}
\end{figure}

This relation corresponds to the following: for $f_0\in\homo_{\bbb}(\lambda_0,\lambda_1), f_1\in\homo_{\bbb}(\lambda_1,\lambda_2)$ and $f_2\in\homo_{\bbb}(\lambda_2,\lambda_0)$ we have: $f_0^\star\otimes f_1\otimes f_2=f_0\otimes f_2^\star\otimes f_1^\star$. 
\\\\
This shows that $\ddd^+_\Lambda(-,1)\otimes_{\ddd^+_{\Lambda,\text{open}}}\bbb$ is isomorphic, as a vector space, to the quotient of 
$$\bigoplus_n\left( \bigoplus_{\lambda_0,\dots, \lambda_{n-1}}\homo_{\bbb}(\lambda_0,\lambda_1)\otimes\dots\otimes\homo_{\bbb}(\lambda_{n-1},\lambda_0)\right),$$
by the relation $f_0^\star\otimes g=f_0\otimes g^\star$, modulo the subspace spanned by the elements of the form $f_0\otimes \dots \otimes f_{n-1}$ above. 
This is precisely the definition given for the normalized involutive Hochschild complex $\overline{C}^{\text{inv}}_\bullet(\bbb)$, ignoring the differential $d$ momentarily. The compatibility with $d$ follows from Proposition 7.4.3 \cite{Costello07}. \qedhere
\end{proof}



\bibliographystyle{amsalpha}
\bibliography{/Users/Ramses/Documents/Universitat/10_Paper_AGT/Bibliografia}

\end{document}